\documentclass[11pt]{article}
\usepackage{epsfig}
\usepackage{amsmath,amssymb,amsthm,amsthm}
\usepackage{mathtools} 
\usepackage{mathrsfs}   
\usepackage[utf8]{inputenc}
\usepackage{hyperref}
\usepackage{tikz,eucal,enumerate,mathrsfs,bbm,verbatim}
\usepackage{chngcntr}

\usepackage[normalem]{ulem}
\usepackage{cancel}

\begin{document}
%
\newcounter{theorem}

\newtheorem{theorem}[theorem]{Theorem}
\newtheorem{lemma}[theorem]{Lemma}
\newtheorem{thm}[theorem]{Theorem}
\newtheorem{prop}[theorem]{Proposition}
\newtheorem{defn}[theorem]{Definition}
\newtheorem{cor}[theorem]{Corollary}

\numberwithin{equation}{section}
\numberwithin{theorem}{section}
\counterwithin{figure}{section}

\newcommand{\beqn}{\begin{equation}}
\newcommand{\eeqn}{\end{equation}}
\def\R{{\mathbb R}}

\def\cleq{{\preccurlyeq}}
\def\cgeq{{\succcurlyeq}}
\def\la{{\langle}}
\def\ra{{\rangle}}
\def\pa{{\partial}}
\def\ep{\epsilon}
\def\nn{\nonumber}


\def\ut{{\tilde{u}}}
\def\vt{{\tilde{v}}}
\def\wt{{\tilde{w}}}
\def\ft{{\tilde{f}}}
\def\gt{{\tilde{g}}}

\def\wbv{{\breve{w}}}
\def\Wbv{{\breve{W}}}
\def\phibv{{\breve{\phi}}}
\def\gbv{{\breve{g}}}
\def\rhobv{{\breve{\rho}}}
\def\rhot{{\tilde{\rho}}}
\def\rhob{{\bar{\rho}}}

\def\Db{{\bar{D}}}
\def\Bb{{\bar{B}}}
\def\L {{\cal L}}

\def\F{{\cal F}}

\newcommand{\tre}{\textcolor{red}}
\newcommand{\tbl}{\textcolor{blue}}

\newcommand{\ol}[1]{\overline{#1}}


\title{The fixed angle scattering problem and wave equation inverse problems with two measurements}

\author{
Rakesh\thanks{
Department of Mathematical Sciences, University of Delaware, Newark, DE 19716, USA.
~Email: rakesh@udel.edu
}
\and
Mikko Salo\thanks{
Department of Mathematics and Statistics, University of Jyv{\"a}skyl{\"a}, Jyv{\"a}skyl{\"a}, Finland.
 Email: mikko.j.salo@jyu.fi.
}
 }

\date{March 18, 2019}

\maketitle

\begin{abstract}
We consider two formally determined inverse problems for the wave equation in more than one space dimension.
Motivated by the fixed angle inverse scattering problem, we show that a compactly supported potential is 
uniquely determined by the far field pattern generated by plane waves coming from exactly two opposite directions. This implies 
that a reflection symmetric potential is uniquely determined by its fixed angle scattering data. We also prove a 
Lipschitz stability estimate for an associated problem.
Motivated by the point source inverse problem in geophysics, we show that a compactly supported potential is uniquely 
determined from boundary measurements of the waves generated by exactly two sources - a point source and an 
incoming spherical wave. 
These results are proved using Carleman estimates and adapting the ideas introduced by Bukhgeim and Klibanov on the use of
Carleman estimates for inverse problems.
\end{abstract}

\tableofcontents


\section{Introduction}\label{sec:intro}

Coefficient determination problems for hyperbolic PDEs arise in areas such as geophysics and medical imaging. 
Formally determined problems, that is, problems where the parameter count for the unknown coefficient 
equals the parameter count of the measured data, present special theoretical and computational challenges, 
particularly for problems in more than one space dimension. In this article we discuss a number of longstanding open formally determined problems for hyperbolic PDEs. We obtain uniqueness and stability results for these inverse problems when we have 
data from two measurements or the coefficient is reflection symmetric.

Our results are for the perturbed wave equation with the unknown coefficient associated with zeroth order perturbation of the wave 
operator - we assume the velocity is a constant. This case is relevant in quantum mechanical applications (see \cite{RS19} for a more detailed discussion) and in cases where the sound speed is constant but the material density is variable and unknown. In many applications, the unknown coefficient of interest is associated with the 
non-constant velocity of propagation for the hyperbolic PDE and determining these coefficients is a more difficult problem.

\subsection{The plane wave scattering problems}

Let us first introduce some notation. Given $x \in \R^n$, $n>1$ we may write $x$ as $x=(y,z)$ with $y \in \R^{n-1}, z \in \R$. Further, $e:=e_n=(0,0,\cdots,0,1)$, $B$ will denote the open unit ball, $S$ its boundary, $\Box= \pa_t^2 - \Delta_x$ and
$\nu$ will denote the outward unit normal to the associated surface.

Here are two of the longstanding problems associated with far field patterns. 
Suppose $q(x)$ is a smooth function on $\R^n$, $n>1$, with compact support. Given a unit vector 
$\omega$ in $\R^n$, consider the  IVP with a plane wave source:
\begin{align}
\Box U + qU = 0, & \qquad (x,t) \in \R^n \times \R,
\label{eq:pUde}
\\
U(x,t) = \delta( t - x \cdot \omega), & \qquad x \in \R^n, ~ t < < 0.
\label{eq:pUic}
\end{align}
This was studied in \cite{RU14} and the following proposition is a 
consequence
of the arguments in the proof of Theorem 1 in \cite{RU14}.
%
\begin{prop}\label{prop:pwell}
The IVP (\ref{eq:pUde}), (\ref{eq:pUic}) has a unique distributional solution $U(x,t,\omega)$ given by
\[
U(x,t, \omega) = \delta( t - x \cdot \omega) + u(x,t,\omega) H(t-x \cdot \omega)
\]
where $u(x,t,\omega)$, a smooth function on the region $t \geq x \cdot \omega $, is the 
unique solution of the characteristic IVP
\begin{align}
\Box u + qu = 0, &\qquad (x,t) \in \R^n \times \R, ~~t > x \cdot \omega,
\label{eq:pude}
\\
u(x, x \cdot \omega, \omega) = - \frac{1}{2} \int_{-\infty}^0 q(x + \sigma \omega) \, d \sigma,
& \qquad x \in \R^n,
\label{eq:puchar}
\\
u(x,t, \omega) = 0, & \qquad x \in \R^n, ~x \cdot \omega < t < < 0.
\label{eq:puic}
\end{align}
Also, for any real $T$, on the region $\{ (x,t) : x \cdot \omega \leq t \leq T \}$, $|u(x,t,\omega)|$ is bounded above
by a continuous function of $\|q\|_{C^{n+4}}$.
Further, when $n=3$, given a unit vector $\theta \in \R^n$ and a real number $s$, we have (as distributions in $s$)
\[
\lim_{r \to \infty} r u(r \theta, r - s, \omega) 
= \frac{1}{2 \pi} \int_{x \cdot \theta=\tau} u_t(x,\tau -s,\omega) \, dS_x
\]
for any $ \tau >0$ for which the support of $q$ is in the region $ x \cdot \theta < \tau$.
\end{prop}
%
%

We mention that $u(x,t,\omega)$ is also characterized as $u(x,t,\omega) = v|_{\{ t \geq x \cdot \omega \}}$ where $v$ solves 
the inhomogeneous PDE $\Box v + qv = -q \delta(t- x \cdot \omega)$ in $\R^n \times \R$ with $v|_{t < < 0} = 0$. However, 
since $v$ vanishes in $\{ t < x \cdot \omega \}$ by finite speed of propagation, only the 
behaviour in the set $\{ t \geq x \cdot \omega \}$ will play a role. For the proofs of the main results it will be natural to work in the 
region $\{ t \geq x \cdot \omega \}$, 
and hence the proposition is stated in this setting. The proof of the upper bound on $|u(x,t,\omega)|$ is not covered in the proof 
of Theorem 1 in \cite{RU14} and we postpone its 
proof to subsection \ref{subsec:pwell}. Also, the upper bound given is not optimal but adequate for our purposes.

Motivated by Proposition \ref{prop:pwell}, for $n=3$, we define the far field pattern of $u(x,t,\omega)$ in the direction $\theta$, with delay $s$, as
\[
\alpha(\theta, \omega,s) := \frac{1}{2 \pi} \int_{x \cdot \theta=\tau} u_t(x,\tau -s,\omega) \, dS_x, \qquad |\theta|=1, 
~ |\omega|=1, ~ s \in \R.
\]
This definition can be extended to all odd dimensions $n \geq 3$ \cite{MU08}. It is closely related to other definitions of far field patterns in 
scattering theory; please see \cite{RU14} for a discussion about this.

There are two longstanding open problems in scattering theory, {\bf the backscattering problem}, consisting 
of examining the injectivity, stability and inversion of the map
\[
q \to \alpha( -\omega, \omega, s)|_{|\omega|=1, ~ s \in \R} \qquad \text{(backscattering problem)}
\]
and the {\bf fixed angle scattering problem} (also called the {\bf single scattering problem}), consisting of examining, {\em for a fixed $\omega$}, 
the injectivity, stability and inversion of the map
\[
q \to \alpha( \theta, \omega, s)|_{|\theta|=1, ~ s \in \R} \qquad \text{(fixed angle scattering problem)}.
\]
These problems are often formulated in terms of the scattering amplitude $a_q(k,\omega,\theta)$, where $k > 0$ is a frequency, which appears in stationary scattering theory (relations between the time domain and stationary approaches are discussed in \cite{Uh01}). Both these problems remain open, including the injectivity of these maps, but there are partial results for both these problems.

For the backscattering problem, the map has been shown to be analytic, shown to be injective when $q$ is small 
enough in some norm or when $q$ is restricted to angularly controlled perturbations of a single $q_0$. Further, 
it has been shown that one can recover the principal singularities of $q$. We only mention here the works \cite{ER92, GU93, MU08, OPS01, RU14, RR12, St90}, and refer to the introduction of \cite{RU14} for further references and discussion. However, for the backscattering problem, one does not even 
know whether the backscattering data is enough to distinguish between zero and non-zero $q$, that is, whether 
$\alpha(-\omega, \omega,s)|_{|\omega|=1, s \in \R} =0$ implies $q=0$.

For the fixed angle scattering problem, uniqueness is known for potentials that are small or belong to a generic set, the principal 
singularities can be recovered, and the zero potential can be distinguished. See \cite{BLM89, St92, Ru01, BCLV18} and further 
results and 
references in \cite{Me18}.  Ramm, in \cite{Ra11}, claims to prove uniqueness for the fixed angle scattering problem for real valued 
smooth compactly supported potentials. However, it was pointed out to us (personal communication) that the 
proof in \cite{Ra11} is incorrect since Lemma 3.1 in \cite{Ra11} contradicts a consequence of the Paley-Wiener theorem.

For the fixed angle scattering problem for $n=3$, without loss of generality, we take $\omega=e$, so the fixed angle  
scattering problem
consists of examining the injectivity, stability, and inversion of the map
\[
q \to \alpha(\theta,e,s)|_{|\theta|=1, ~ s \in \R}.
\]
Since $q$ is compactly supported, we assume that $q$ is supported in $\Bb$. From Proposition
\ref{prop:pwell}, the single scattering problem is equivalent to the recovery of $q$ from the Radon transform (in $x$) of
$u_t(x, t, e)$ over the planes $x \cdot \theta = \tau$, $ \tau \geq 1$, for all $t \in \R$. Since $u_t(x,t,e)$ is compactly
supported for each fixed $t$, from Helgason's Support Theorem (see \cite{He11}), the problem is equivalent 
to recovering
$q$, given $u_t(x,t,e)$ for all $t \in \R$ and all $x$ such that $|x| \geq 1$. Now, from (\ref{eq:pude}) - (\ref{eq:puic}), that $q=0$ outside $B$ and the observation that the characteristic BVP
\begin{align*}
\Box u =0, & \qquad (x,t) \in  (\R^3 \setminus B) \times \R, ~ \tbl{t > x \cdot e},
\\
u = f, & \qquad \text{on} ~ S \times \R  \cap \{ t \geq x \cdot e \},
\\
u =0, & \qquad \tbl{x \cdot e < t} < < 0,
\end{align*}
is well posed, knowing $u(x,t,e)$ on $(S \times \R) \cap \{ t \geq x \cdot e \}$ is equivalent to knowing 
$u(x,t,e)$ on $((\R^3 \setminus B) \times \R) \cap \{ t \geq x \cdot e \}$. 
Hence the fixed angle scattering problem 
is equivalent to studying the injectivity, stability and inversion of the map
\beqn
q \to u(x,t,e)|_{(x,t) \in (S \times \R) \cap \{ t \geq x \cdot e \}}.
\label{eq:temppprob}
\eeqn
Of course Helgason's Support Theorem is an injectivity result and the associated stability estimates are weak 
so the equivalence stated above is only formal, as far as stability and inversion is concerned. For use below, we note 
that the map (\ref{eq:temppprob}) makes sense for any $n>1$.

A problem close to the fixed angle scattering problem is of interest in geophysics. Suppose $q(y,z)$ is a smooth function
on $\R^n$ with support in $z \geq 0$. If $U(x,t,e)$ is the solution of 
the IVP (\ref{eq:pUde}), (\ref{eq:pUic}) with $\omega=e$, then geophysicists make measurements only on
$z=0$ and they are interested in the inversion of the map
\beqn
q \to U(y, z=0,t,e)|_{(y,t) \in \R^{n-1} \times \R}.
\label{eq:geo}
\eeqn
This problem is still open but there are partial results for this problem. When $q$ depends only on $z$, the problem 
is a well understood one dimensional inverse problem for a hyperbolic PDE thanks to the 
work of Gelfand, Levitan, Krein and others - see \cite{Sy86} for a survey of the results. For the 
multidimensional problem, in \cite{SS85}, Sacks and Symes showed that this map is differentiable and its derivative is injective at points $q(\cdot)$ where $q$ depends only on $z$. In \cite{Ro02}, Romanov showed this map is 
invertible and constructed its inverse if the domain of this map was restricted to $q(y,z)$ which are analytic in $y$.

If $q$ is compactly supported, say $q$ is supported in $\Bb$, in addition to being supported in $z \geq 0$, then 
the study of the map (\ref{eq:geo}) is closely related to the study of the map (\ref{eq:temppprob}).
Since $q=0$ on $z \leq 0$, from (\ref{eq:pude}) - (\ref{eq:puic}) and the well-posedness of the characteristic
BVP
\begin{align*}
\Box u =0, & \qquad y \in \R^{n-1}, ~ z < 0, ~ t > z,
\\
u(y,z=0,t) = f, & \qquad (y,t) \in \R^{n-1} \times [0,\infty),
\\
u(y,z,z) =0, & \qquad z \leq 0, ~ y \in \R^2,
\end{align*}
we conclude that knowing $u(y,z=0,t)$ for all $(y,t) \in \R^{n-1} \times [0,\infty)$ 
we can determine $u_z(y,z=0,t)$ for all $(y,t) \in \R^{n-1} \times [0,\infty)$ - actually one can write
an explicit formula using the fundamental solution of the wave operator. 
Finally, since $q=0$ outside $B$,
from (\ref{eq:pUde}) - (\ref{eq:pUic}) and Holmgren's theorem on unique continuation for
\begin{align*}
\Box U =0, & \qquad (y,z,t) \in \R^n \times \R, ~ z > 0, ~ |(y,z)| > 1,
\\
U(y,z=0,t) = f, ~ U_z(y,z=0,t) = g, & \qquad (y,t) \in \R^{n-1} \times \R,
\end{align*}
we conclude that knowing $U(y,0,t), U_z(y,0,t)$ for $(y,t) \in \R^{n-1} \times \R$, uniquely determines
$U(y,z,t)$ on $S \times \R$. Hence, this geophysics problem is equivalent to the study of the 
map (\ref{eq:temppprob}).

The injectivity and stability of the fixed angle scattering map (\ref{eq:temppprob}) remains open but we show stability if we have data from two experiments; we show that the map
\beqn
q \to [u(x,t,e)|_{\Sigma_l}, u(x,t,-e)|_{\Sigma_r} ]
\label{eq:pmap}
\eeqn
is injective and its inverse is Lipschitz stable in certain norms; here
\[
\Sigma_l = \{ (y,z,t) \in S \times (-\infty,T] : t \geq z \}, \qquad \Sigma_r = \{ (y,z,t) \in S \times (-\infty,T] : t \geq -z \}.
\]
%

\begin{theorem}[Two plane wave data]\label{thm:twop1}
Suppose $q_i$, $i=1,2$ are smooth functions on $\R^n$ with support in 
$\Bb$ and $u_i(x,t,e)$, $u_i(x,t,-e)$ the solutions of (\ref{eq:pude}) -(\ref{eq:puic})) with $q=q_i$ and
$\omega=e, -e$. If $T>6$ and $\|q_i\|_{C^{n+4}}  \leq M$, $i=1,2$, then
\[
\|q_1 -q_2\|_{L^2(B)}
\cleq
\| (u_1-u_2)(\cdot, \cdot,e)\|_{H^1(\Sigma_l)} + 
\| (u_1-u_2)(\cdot, \cdot,-e)\|_{H^1(\Sigma_r)}
+ \| (u_1-u_2)(\cdot,\cdot,e) \|_{H^1(\Sigma_l \cap \{t=z\})}
\]
with the implied constant determined by $T$ and $M$. 
\end{theorem}

A corollary of Theorem \ref{thm:twop1} is a result for single measurement data provided $q$ is compactly supported
and an even function in $z$, or more generally, for a fixed incoming direction $\omega$, $q$ is symmetric about the plane
$x \cdot \omega =c$ for some $c$. If $q(y,z)$ is an even function of $z$ and $u(x,t,e)$
$u(x,t,-e)$ are the solutions of (\ref{eq:pude}) - (\ref{eq:puic}) for $\omega=e$ and $\omega=-e$ then one observes that
\[
u(y,-z,t,e) = u(y,z,t,-e), \qquad \forall (y,z,t) \in \R^n \times \R, ~ t \geq -z
\]
hence knowing $u(\cdot,\cdot,e)|_{\Sigma_l}$ is equivalent to knowing $u(\cdot,\cdot,-e)|_{\Sigma_r}$, so Theorem
\ref{thm:twop1} implies the following corollary.
\begin{cor}[Fixed angle scattering for symmetric potentials]\label{cor:even}
Suppose $q_i$, $i=1,2$ are smooth functions on $\R^n$ with support in 
$\Bb$ and $u(x,t,e)$ the solution of (\ref{eq:pude}) -(\ref{eq:puic}) with $q=q_i$ and
$\omega=e$. If $q_i(y,z) = q_i(y,-z)$ for all $(y,z) \in \R^n$, $i=1,2$, if $T>6$, and if 
$\| q_i \|_{C^{n+4}} \leq M$,  $i=1,2$, then
\[
\|q_1 -q_2\|_{L^2(B)}
\cleq
\| (u_1-u_2)(\cdot, \cdot,e)\|_{H^1(\Sigma_l)} 
+ \| (u_1-u_2)(\cdot,\cdot,e) \|_{H^1(\Sigma_l \cap \{t=z\})}
\]
with the implied constant determined by $T$ and $M$. 
\end{cor}
Recently, in \cite{RS19}, we have improved upon the result in Corollary \ref{cor:even} and proved stability for the fixed 
angle scattering problem under even, odd or $y$-controlled perturbations.


\subsection{The point source and spherical wave source problems  }

 Consider the following IVP problem associated to a point source
\begin{align}
\Box V + q V = \delta(x,t), & \qquad (x,t) \in \R^3 \times \R,
\label{eq:sVde}
\\
V =0, & \qquad t <0.
\label{eq:sVic}
\end{align}
This problem has been studied in \cite{Ro74} and elsewhere and the following is a consequence of the results in \cite{Bl17}.
\begin{prop}\label{prop:sVwell}
If $q(x)$ is a compactly supported smooth function on $\R^3$ which is zero in a neighborhood of the origin then 
(\ref{eq:sVde}), (\ref{eq:sVic}) has a unique distributional solution given by 
\[
V(x,t) = \frac{1}{4 \pi} \frac{\delta(t-|x|)}{|x|} + v(x,t)H(t-|x|)
\]
where $v(x,t)$ is a smooth function on $t \geq |x|$ (see Figure \ref{fig:vwdef})  and is the 
unique solution of the Goursat problem
\begin{align}
\Box v +qv =0, & \qquad t > |x|,
\label{eq:svde}
\\
v(x,|x|) =  -\frac{1}{8 \pi} \int_0^1 q(sx) \, ds, & \qquad x \in \R^3. 
\label{eq:svcc}
\end{align}
\end{prop}
We take $q=0$ in a neighborhood of the origin because the behavior of $v(x,t)$ is subtle near $x=0, t=0$ if $q(0) \neq 0$. 
Further, we need this assumption for the result stated below. 

Another longstanding open problem, which we call the {\bf point source problem},
is the injectivity, stability and inversion of the map
\[
q \to v|_{S \times [1,\infty)}. 
\]
Romanov has observed that $v|_{S \times [1,\infty)} =0$ implies $q=0$ and several people 
have observed that the map is injective for small $q$ - see \cite{RS11}.

\begin{figure}
\begin{center}
\epsfig{file=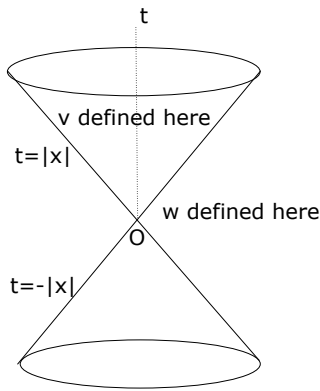, height=2in}
\end{center}
\caption{Domains of $v$ and $w$}
\label{fig:vwdef}
\end{figure}

Next we describe another inverse problem.  Consider the following
IVP associated to an incoming spherical wave
\begin{align}
\Box W + q W = 0, & \qquad (x,t) \in \R^3 \times \R,
\label{eq:sWde}
\\
W =  \frac{1}{4 \pi} \frac{\delta(t+|x|)}{|x|}, & \qquad x \in \R^3, ~~ t <-1.
\label{eq:sWic}
\end{align}
We show the following regarding the solution of this IVP.
\begin{prop}\label{prop:sWwell}
Suppose $q$ is a compactly supported smooth function on $\R^3$ with $q=0$ in a neighborhood of the origin.
The IVP (\ref{eq:sWde}), (\ref{eq:sWic}) has a unique distributional solution 
which is smooth on the region $t \neq \pm |x|$. Further, on the region $t<|x|$, $W(x,t)$ may be expressed as
\[
W(x,t) = \frac{1}{4 \pi} \frac{\delta(t+|x|)}{|x|} + w(x,t) H(t+|x|)
\]
where $w(x,t)$ is a smooth function on the region $-|x| \leq t < |x|$ and satisfies
\begin{align}
\Box w + qw =0, & \qquad (x,t) \in \R^3 \times \R, ~ -|x| < t < |x|, ~ x \neq 0
\label{eq:swde}
\\
w(x,-|x|) =  - \frac{1}{8 \pi} \int_1^\infty q(sx) \, ds, & \qquad x \in \R^3, ~ x \neq 0
\label{eq:swcc}
\\
w(x,t)=0, & \qquad -|x| \leq t < -1.
\label{eq:swic}
\end{align}
\end{prop}
The behavior of $W$ above the upper cone $t = |x|$ is subtle and perhaps not well understood.

Another open problem, proposed by Romanov, which we call the {\bf spherical wave problem}, is the injectivity, stability and 
inversion of the map
\[
q \to w|_{S \times [-1,1)}.
\]
By a unique continuation argument one may observe that $q=0$ if $w|_{S \times [-1,1)}=0$. 
In \cite{Ba18}, it was shown that the map is injective if $q$ is restricted to angularly controlled perturbations of a fixed $q_0$ or 
if $q$ is small in a certain norm. For the two dimensional case, in \cite{Ro02}, it was shown that the map is 
injective and
an inverse may be constructed provided $q(r \theta)$ is restricted to functions which are analytic in $\theta$.
The spherical wave problem remains open in all dimensions greater than $1$.

The spherical wave problem may be regarded as a type of backscattering problem with the difference that the 
data comes from the solution of only one IVP. For the backscattering problem we are given very limited data from each of 
a large number of solutions of the PDE. An inability to fruitfully combine data from many solutions is what makes the
backscattering problem difficult. We believe the spherical wave problem may be an easier problem and its solution may 
provide insight into the solution of the backscattering problem.

The point source problem and the spherical wave problem remain open, but given the data for both problems
the associated map is injective.

%
\begin{theorem}[Point source and spherical source data]\label{thm:pointspherical}
For $q$ which are smooth functions on $\R^3$ with support in $\Bb$ and zero in a neighborhood of the origin, let
$v(x,t)$, $w(x,t)$ be the functions in Proposition \ref{prop:sVwell} and Proposition \ref{prop:sWwell}. The map
\beqn
q \to \left [ v|_{S \times [1,3]}, w|_{S \times [-1,1)} \right ]
\label{eq:psmap}
\eeqn
is injective.
\end{theorem}

The article \cite{La19} describes two other pairs of data which lead to results similar to Theorem \ref{thm:twop1} and 
Theorem \ref{thm:pointspherical} by using Proposition \ref{prop:planeprop} and Proposition \ref{prop:sprop} in our article.

To prove Theorem \ref{thm:twop1} and Theorem \ref{thm:pointspherical} we use 
the two solutions of the PDE to construct an $H^1$ solution of $(\Box + q) \alpha =0$ on a cylindrical region 
$\Bb \times [T_1, T_2]$, for some $T_1, T_2$, such that $\alpha$ restricted to a characteristic surface
(either $t=z$ or $t=|x|$) in the interior
of the region  is an integral of $q$. Then we 
adapt the technique in \cite{IY01} to prove stability for certain hyperbolic inverse problems.  
In \cite{IY01} (see \cite{BY17} for a better organized presentation), 
the $\alpha|_{t=0}$ is related to $q$ where as, for our problems, $\alpha$ restricted to a characteristic surface is related to $q$; 
thus a need to adapt the technique in \cite{IY01}. 

The technique in \cite{IY01} is itself a modification of the breakthrough ideas, introduced in \cite{BK81}, for solving formally determined 
inverse problems for 
hyperbolic and parabolic PDEs. However, the 
problem studied in \cite{BK81} and \cite{IY01} required a source in the form of an initial condition $\alpha(x,0) = f(x)$ with
$f(x)>0$ at each 
point on the domain. In geophysical and some other applications, such sources are difficult to generate and the 
preferred source is an impulsive source such as a point source $\alpha(x,0) = \delta(x)$ or a plane wave source 
$\alpha(x,t) = \delta(t-x \cdot \omega)$ for
$t<<0$. Our results are for these impulsive sources in space dimension greater than one, for which there are just a few results
- we have mentioned some results earlier and \cite{Kl05} is interesting. 
For a survey of the results for problems associated with a source of the form 
$\alpha(x,0) = f(x)$ with $f(x)>0$ at each point in the domain,
as in \cite{Be04}, \cite{Kh89}, \cite{IY01}, \cite{SU13} and several other articles, we refer the reader to 
\cite{BY17},  \cite{Kl13}, \cite{Bu00} and \cite{Is06}.


We use the following notation through out this article.
 Given $x \in \R^n$, $n>1$ we may write $x$ as $x=(y,z)$ with $y \in \R^{n-1}, z \in \R$, or we may write
a non-zero $x$ as $x = r \theta$ where $r=|x|$ and $\theta=x/|x|$. We define the radial and angular derivatives
\[
\pa_r = \frac{x}{|x|} \cdot \nabla, 
\qquad \Omega_{ij} = x_i \pa_j  -x_j \pa_i, \qquad i,j=1,2, \cdots, n, ~ i \neq j
\]
and note that
\beqn
|\nabla f(x)|^2 = |f_r|^2 + \frac{1}{2r^2} \sum_{i \neq j} |\Omega_{ij} f(x)|^2.
\label{eq:fangular}
\eeqn
Further $e:=(0,0,\cdots,0,1)$,
$B$ will denote the open unit ball, $S$ its boundary  and $\Box= \pa_t^2 - \Delta_x$. We say the map 
$f:X\to Y$ is stable if $f$ is injective and its inverse is locally Lipschitz for some norms on $X$ and $Y$. 
We say $ a \cleq b$ if $a \leq C b$ for some constant $C$. 

If $k$ is a non-negative integer, $A$ is the closure of a bounded open subset of $\R^n$ and $f:A \to \R$ then 
$\|f\|_{H^k(A)}$ will denote the 
Sobolev space norm of $f$ on $A$, and for a fixed weight $\psi$ and $\sigma>0$
\[
\| f\|_{1,\sigma,A} :=  \left ( \int_A e^{2 \sigma \psi} ( |\nabla f|^2 + \sigma^2 |f|^2) \right )^{1/2},
\qquad
\|f\|_{0,\sigma,A} :=  \left ( \int_A e^{2 \sigma \psi} \, |f|^2  \right )^{1/2}.
\]
If $A$ is a bounded hypersurface in $\R^n$ and $f:A \to \R$ then $\|f\|_{H^k(A)}$ will denote the Sobolev space norm
of $f$ on $A$, with derivatives only in directions tangential to $A$, and
\[
\| f\|_{1,\sigma,A} :=  \left ( \int_A e^{2 \sigma \psi} ( |\nabla_A f|^2 + \sigma^2 |f|^2) \right )^{1/2},
\qquad
\|f\|_{0,\sigma,A} :=  \left ( \int_A e^{2 \sigma \psi} \, |f|^2  \right )^{1/2}
\]
where $\nabla_A f$ denotes the gradient of $f$ on $A$ made up only of derivatives of $f$ in directions tangential 
to $A$. Further, for any bounded real valued function $f$, the supremum of $|f|$ will be denoted by $\|f\|_\infty$.
%



\section{Proofs for the two plane wave sources problem}

%
%

We define the following useful subsets of $\R^n \times \R$;
\begin{align*}
Q := \Bb \times [-T,T], & \qquad \Sigma := S \times [-T,T], &  \hspace{-0.5in} \Gamma = Q \cap \{t=z\},
\\
Q_{\pm} := Q \cap \{ \pm(t-z) \geq 0 \},  & \qquad \Sigma_{\pm} = \Sigma \cap \{ \pm (t-z) \geq 0 \}, &
\\
Q_{+,\tau} = Q_+ \cap \{ t \leq \tau \}, & \qquad  \Sigma_{+,\tau} = \Sigma_+ \cap \{t \leq \tau \}, &
\end{align*}
for any $\tau \in \R$.

\subsection{The main proposition for plane wave sources}

The following proposition is crucial in the proof of Theorem \ref{thm:twop1}. We postpone its proof to 
subsection \ref{subsec:planeprop}.
Suppose $p(x)$ and $q(x)$ are smooth functions on $\R^n$, $n>1$, with support in $\Bb$. Define
\[
\rhob(y) = \int_\R p(y,z) \, dz
\]
and let $\rho(x,t) \in H^1(Q_-)$ be the solution, guaranteed by Lemma \ref{lemma:pweak}, of the characteristic BVP
\begin{align*}
(\Box + q) \rho =0, & \qquad \text{on} ~ Q_-,
\\
\rho(y,z,z) = \rhob(y), & \qquad (y,z) \in \Bb,
\\
\rho(y,z,t) = \rhob(y), & \qquad (y,z,t) \in \Sigma_-.
\end{align*}
%
\begin{prop}\label{prop:planeprop}
Suppose $p$, $q$ and $\rho$ are as above, $f$ a bounded function on $Q$, and $\alpha \in H^1(Q)$ with
$\alpha$ smooth on $Q_+$, $\alpha + \rho$ smooth on $Q_-$, $ \alpha|_\Sigma \in H^1(\Sigma), 
~ \pa_\nu \alpha \in L^2(\Sigma)$. If 
\begin{align}
(\Box + q) \alpha(x,t) = p(x) f(x,t), & \qquad (x,t) \in Q, ~~ (\text{as distributions})
\label{eq:alphade}
\\
\alpha(y,z,z) = \int_{-\infty}^z p(y,s) \,ds, & \qquad (y,z) \in \Bb,
\label{eq:alphacc}
\end{align}
then
\beqn
\| p\|_{L^2(\Bb)} \cleq \| \alpha \|_{H^1(\Sigma)} + \| \pa_\nu \alpha\|_{L^2(\Sigma)}
+ \|\rho\|_{H^1(\Sigma \cap \Gamma)},
\label{eq:alphaineq}
\eeqn
provided $T>6$. Here the constant depends only on $\|q\|_\infty$, $\|f\|_\infty$ and $T$.
\end{prop}

%
%

\subsection{Proof of Theorem \ref{thm:twop1}}

Suppose $q_i$, $i=1,2$ are smooth functions on $\R^n$ with support in $\Bb$, $u_i(x,t,e)$, $u_i(x,t,-e)$ the 
corresponding solutions of (\ref{eq:pude})-(\ref{eq:puic}) for $\omega=e,-e$. Then, by Proposition \ref{prop:pwell}, 
\begin{align*}
(\Box + q_1)( u_1(x,t,e) - u_2(x,t,e)) = - (q_1-q_2)(x) u_2(x,t,e), & \qquad (x,t) \in \R^n \times \R \cap  \{  z \leq t  \},
\\
(u_1 - u_2)(y,z,z,e) = -\frac{1}{2}  \int_{-\infty}^z (q_1 - q_2) (y,s) \, ds,  & \qquad (y,z) \in \R^n,
\\
(u_1 - u_2)(x,t,e) = 0, & \qquad (x,t) \in \R^n \times \R \cap \{z \leq t < -1 \},
\end{align*}
and
\begin{align*}
(\Box + q_1)( u_1(x,t,-e) - u_2(x,t,-e)) = - (q_1-q_2)(x) u_2(x,t,-e), & \qquad (x,t) \in \R^n \times \R \cap \{ -z \leq t  \},
\\
(u_1 - u_2)(y,z,-z,-e) = -\frac{1}{2}  \int_z^\infty (q_1 - q_2) (y,s) \, ds, & \qquad (y,z) \in \R^n,
\\
(u_1 - u_2)(x,t,-e) = 0, & \qquad (x,t) \in \R^n \times \R \cap \{-z \leq t < -1 \}.
\end{align*}
Define
\[
\rhob(y) =  \frac{1}{2} \int_\R (q_1 - q_2)(y,z) \, dz, \qquad y \in \R^{n-1},
\]
and let $\rho \in H^1(Q_-)$ be the solution, guaranteed by Lemma \ref{lemma:pweak}, of the characteristic BVP
\begin{align*}
(\Box + q_1) \rho = 0, & \qquad (x,t) \in Q_-,
\\
\rho(y,z,z) = \rhob(y), & \qquad (y,z) \in \Bb,
\\
\rho(y,z,t) = \rhob(y), & \qquad (y,z,t) \in \Sigma_-.
\end{align*}
For $(x,t) \in Q$, define
\[
\ut_2(x,t) = \begin{cases} u_2(x,t,e), & \qquad t \geq z, \\  - u_2(x,-t,-e), & \qquad t<z, \end{cases}
\]
and 
\[
\alpha(x,t) = \begin{cases} (u_1-u_2)(x,t,e), & \qquad t \geq z, \\  - (u_1 - u_2)(x,-t,-e)  - \rho(x,t), & \qquad t<z .\end{cases}
\]
The function $\alpha$ combines the measurements associated with the directions $e$ and $-e$, and the function $\rho$ has 
been subtracted in $Q_-$ so $\alpha$ is $H^1$ across the plane $ t=z$. The function $\rho$ is 
only required for the stability estimate; if one is only interested in a uniqueness result for the inverse problem, the function $\rhob$ 
will be zero if the data for $q_1$ and $q_2$ agree, hence $\rho \equiv 0$ in this case.

Using Lemma \ref{lemma:pweak}, we see that for $(x,t) \in Q$ we have in the sense of distributions 
\begin{align*}
(\Box + q_1) \alpha (x,t) & = (\Box + q_1) ( (u_1-u_2)(x,t,e) H(t-z) - (u_1 - u_2)(x,-t,-e) H(z-t) )
\\
& = - (q_1-q_2) \ut_2 + 2 \delta(t-z)  (\pa_t + \pa_z) (u_1 - u_2)(x,t,e)
\\
& \qquad  + 2 \delta(t-z) (\pa_t + \pa_z) ( ( u_1 - u_2)(x,-t,-e) )
\\
& =  - (q_1-q_2) \ut_2 + 2 \delta(t-z)  (\pa_t + \pa_z) (u_1 - u_2)(x,t,e) 
\\
& \qquad + 2 \delta(t-z) (-\pa_t + \pa_z) ( u_1 - u_2)(x,-t,-e) 
\\
& =  - (q_1-q_2) \ut_2 + 2 \delta(t-z)  (\pa_t + \pa_z) (u_1 - u_2)(y,z,z,e) 
\\
& \qquad + 2 \delta(t-z) (-\pa_t + \pa_z) ( u_1 - u_2)(y,z,-z,-e) 
\\
& = - (q_1-q_2) \ut_2 +2 \delta(t-z) \frac{d}{dz} ( (u_1-u_2)(y,z,z,e) + (u_1-u_2)(y,z,-z,-e) )
\\
& =  - (q_1-q_2) \ut_2- \delta(t-z) \frac{d}{dz} \left ( \int_{-\infty}^z q(y,s) \, ds + \int_z^\infty q(y,s) \, ds \right )
\\
& = - (q_1-q_2)\ut_2 - \delta(t-z) \frac{d}{dz} \int_\R q(y,s) \, ds 
\\
& =  - (q_1-q_2)(x) \ut_2(x,t).
\end{align*}
Further, for $(y,z) \in \Bb$, we have
\begin{align*}
\alpha(y,z,z+) & = (u_1 - u_2)(y,z,z,e) = - \frac{1}{2} \int_{-\infty}^z (q_1 - q_2)(y,s) \, ds
\\
\alpha(y,z,z-) & = - (u_1 - u_2)(y,z,-z,-e) - \rho(y,z,z)
 = \frac{1}{2} \int_z^\infty (q_1 - q_2)(y,s) \, ds - \frac{1}{2}  \int_\R (q_1 - q_2)(y,s) \, ds
\\
& = - \frac{1}{2} \int_{-\infty}^z (q_1 -q_2)(y,s) \, ds;
\end{align*}
so
\[
\alpha(y,z,z+) = \alpha(y,z,z-).
\]
Hence, $\alpha \in H^1(Q)$ with
\begin{align*}
(\Box + q_1) \alpha = -(q_1-q_2) \ut_2, & \qquad (x,t) \in Q,
\\
\alpha(y,z,z) = - \frac{1}{2} \int_{-\infty}^z (q_1 - q_2)(y,s) \, ds, & \qquad (y,z) \in \Bb.
\end{align*}
Further  $\alpha|_\Sigma \in H^1(\Sigma)$ and $\pa_\nu \alpha \in L^2(\Sigma)$ - we verify this at the end of this proof.
So from Proposition \ref{prop:planeprop} we have
\begin{align}
\| q_1 - q_2 \|_{L^2(B)} 
& \cleq \| \alpha \|_{H^1(\Sigma)} + \|\pa_\nu \alpha \|_{L^2(\Sigma)}
+  \| \rho \|_{H^1(\Sigma \cap \Gamma)}
\label{eq:pproptemp}
\end{align}
with the constant dependent only on the supremum of  $\ut_2$ on $Q$, $\|q_1\|_\infty$ and $T$, hence dependent 
only on $T$ and $\|q_i\|_{C^{n+4}}$,  $i=1,2$.
Using the definition of $\alpha$ and Lemma \ref{lemma:pexterior} (together with the analogue of Lemma \ref{lemma:pexterior} in $\{ t < z \}$) we have
\begin{align}
\|\pa_\nu \alpha \|_{L^2(\Sigma)} 
& \cleq \| \pa_{\nu} ( u_1 - u_2) (\cdot, \cdot, e) \|_{L^2(\Sigma_+)}
+ \| \pa_{\nu} ( u_1 - u_2) (\cdot, \cdot, -e) \|_{L^2(\Sigma_-)}
+ \|\pa_\nu \rho \|_{L^2(\Sigma_-)}
+ \|\rho\|_{H^1(\Sigma \cap \Gamma)} 
\nn
\\
& \cleq \|\alpha\|_{H^1(\Sigma_+)} +  \|\alpha\|_{H^1(\Sigma \cap \Gamma)}
+ \|\alpha + \rho \|_{H^1(\Sigma_-)}  + \|\alpha + \rho \|_{H^1(\Sigma \cap \Gamma)}
+ \| \pa_\nu \rho \|_{L^2(\Sigma_-)}
+ \|\rho\|_{H^1(\Sigma \cap \Gamma)} 
\nn
\\
& \cleq  \|\alpha\|_{H^1(\Sigma)} 
+  \| \alpha \|_{H^1(\Sigma \cap \Gamma)} +  \| \rho \|_{H^1(\Sigma_-)} + \| \rho \|_{H^1(\Sigma \cap \Gamma)}
+  \| \pa_\nu \rho \|_{L^2(\Sigma_-)}
\nn
\\
& \cleq  \|\alpha \|_{H^1(\Sigma)} 
+  \| \alpha \|_{H^1(\Sigma \cap \Gamma)} + \| \rho \|_{H^1(\Sigma \cap \Gamma)},
\label{eq:temparho}
\end{align}
where we used Lemma \ref{lemma:pweak} in the last step.
For $(y,z,z) \in \Sigma \cap \Gamma $ we have 
\begin{align*}
\rho(y,z,z) & = \rhob(y)  = \frac{1}{2} \int_\R q_1(y,s) \, ds - \frac{1}{2} \int_\R q_2(y,s) \, ds
\\
& =  \frac{1}{2} \int_{-\infty}^{\sqrt{1-|y|^2}}  q_1(y,s) \, ds  - \frac{1}{2} \int_{-\infty}^{\sqrt{1-|y|^2}}  q_2(y,s) \, ds
\\
& = (u_2 - u_1)(y, \sqrt{1-|y|^2}, \sqrt{1-|y|^2}, e)
\\
& = (u_2 - u_1)(y, |z|, |z|,e)
\\
& =-\alpha(y,|z|,|z|)
\end{align*}
Hence
\beqn
\|\rho\|_{H^1(\Sigma \cap \Gamma)} \cleq \|\alpha\|_{H^1( \Sigma \cap \Gamma)}
\label{eq:tempa1b1}
\eeqn
and using this in (\ref{eq:temparho}) we obtain
\[
\|\pa_\nu \alpha \|_{L^2(\Sigma)}  \cleq  \|\alpha \|_{H^1(\Sigma)} 
+  \| \alpha \|_{H^1(\Sigma \cap \Gamma)}.
\]
Inserting these estimates in \eqref{eq:pproptemp} and using  (\ref{eq:tempa1b1}) and Lemma \ref{lemma:pweak}
we have
\begin{align*}
\| q_1 - q_2 \|_{L^2(B)} 
&  \cleq \| \alpha \|_{H^1(\Sigma)} +  \| \alpha \|_{H^1(\Sigma \cap \Gamma)}
\\
& 
\cleq  \|\alpha \|_{H^1(\Sigma_+)} +  \| \alpha + \rho \|_{H^1(\Sigma_-)} + 
\|\rho\|_{H^1(\Sigma_-)} +  \| \alpha \|_{H^1(\Sigma \cap \Gamma)}
\\
& 
\cleq  \|\alpha \|_{H^1(\Sigma_+)} +  \| \alpha + \rho \|_{H^1(\Sigma_-)} +  \| \alpha \|_{H^1(\Sigma \cap \Gamma)},
\end{align*}
and the theorem is proved except for the verification of the claims 
$\alpha|_\Sigma \in H^1(\Sigma)$ and $\pa_\nu \alpha \in L^2(\Sigma)$.

Now, by definition, $\alpha|_{\Sigma_+}$, $(\alpha+\rho)|_{\Sigma_-}$ and $\rho|_{\Sigma_-}$
are smooth and $\alpha$ is continuous across $\Gamma \cap \Sigma$, hence $\alpha|_\Sigma \in H^1(\Sigma)$.
Further, $(\Box+q_1)\alpha = - (q_1-q_2) \tilde{u}_2 \in L^2(Q)$, $\alpha$ is smooth near $t=T$ and $\alpha \in H^1(Q)$,
so $\alpha$ is a solution of a backward IBVP for a hyperbolic PDE with RHS in $L^2$, smooth initial data, and $H^1$ Dirichlet boundary 
data, so $\pa_\nu \alpha \in L^2(\Sigma)$ by Theorem 3.1 in \cite{BY17}.

%
%

\subsection{Proof of Proposition \ref{prop:planeprop}}\label{subsec:planeprop}

\begin{figure}[h]
\begin{center}
\epsfig{file=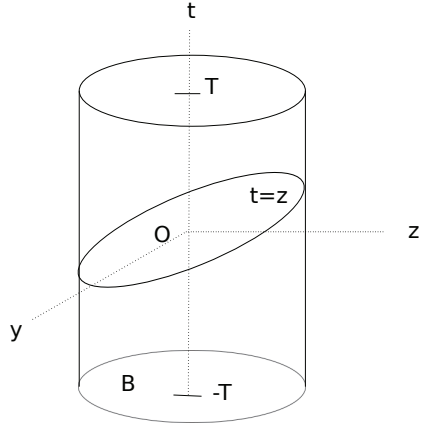, height=2.2in}
\end{center}
\label{fig:zt}
\caption{The regions $Q$ and $Q_{\pm}$}
\end{figure}

For any $a > 1$ define
\[
\phi(y,z,t) := 5(a-z)^2 + 5 |y|^2 - (t-z)^2.
\]
From Lemma \ref{lemma:pweight}, for large enough $\lambda$, 
\[
\psi(y,z,t) = e^{\lambda \phi(y,z,t)}
\]
is strongly pseudoconvex (Definition 1.1 in \cite{Ta96}) w.r.t $\Box + q$ in a neighborhood of $Q$. 

Since $T>6$, we claim there is an $a>1$ such that the smallest value of $\phi$ on $Q \cap \{ t=z \}$ is strictly 
larger than the largest value of $\phi$ on $Q \cap \{ |t|=T\}$. The largest value of $\phi$ on $Q \cap \{ |t|=T\}$ is bounded
above by $ 5(a+1)^2 + 5 - (T-1)^2$ and the smallest value of $\phi$ on $Q \cap \{ t=z \}$ is $5 (a-1)^2$, so we want
\[
5(a-1)^2  > 5 (a+1)^2 + 5 -  (T-1)^2
\]
which is equivalent to
\[
(T-1)^2 > 20a + 5.
\]
Hence for $T>6$ any $a \in (1,  ( (T-1)^2 - 5)/20 )$ will work. Therefore we can find a $T' \in (1,T)$, $T'$ close to $T$, and 
real $c$ and $\delta>0$ such that 
\vspace{-0.25in}
\begin{itemize}
\item $\psi \leq c - \delta$ on $\Bb \times \{ T' \leq |t| \leq T \}$,
\item $\psi \geq c+\delta$ on $Q \cap \{ t=z \}$.
\end{itemize}

We fix an $a \in (1,  ( (T-1)^2 - 5)/20 )$ and the large enough $\lambda$. Let $\chi(t)$ be a smooth function 
on $\R$ with $\chi =1$ on $|t| \leq T'$ and 
$\chi=0$ on $|t| \geq T$ and define
\[
\beta(x,t) := \chi(t) \, \alpha(x,t), \qquad (x,t) \in Q.
\]
Since $\psi$ is strongly pseudo-convex w.r.t $\Box + q$ near $Q$ and the combined 
Dirichlet and Neumann boundary operators satisfy the strong Lopatinskii condition with respect to $\psi$ and $\Box$ 
(see Definition 1.6 in \cite{Ta96}), that
$\beta \in H^1(Q)$, $(\Box + q) \beta \in L^2(Q)$, $\beta|_\Sigma \in H^1(\Sigma)$,
$\pa_{\nu} \beta|_{\Sigma} \in L^2(\Sigma)$, and $\beta=0$ near 
$Q \cap \{ t= \pm T \}$,
from Theorem 1 in \cite{Ta96} we have (for $\sigma$ large enough depending on $\| q \|_{\infty}$ and $T$)
\beqn
\sigma \|\beta\|_{1,\sigma,Q}^2 \cleq \|(\Box + q) \beta \|_{0,\sigma,Q}^2 + \sigma \| \beta\|_{1,\sigma, \Sigma}^2
+ \sigma \| \pa_\nu \beta \|_{0,\sigma, \Sigma}
\label{eq:betacarl}
\eeqn
with the constant dependent only on $T$ and $\|q\|_\infty$.

Now
\begin{align*}
(\Box + q) \beta & = \chi ( \Box + q) \alpha + \chi'' \alpha + 2 \chi' \alpha_t,
\end{align*}
hence
\[
|(\Box + q) \beta| \cleq |(\Box + q) \alpha| + (|\chi'|+|\chi''|)(|\alpha| + |\alpha_t|)
\]
so
\begin{align}
\|(\Box + q) \beta \|_{0, \sigma, Q}^2 
& \cleq \int_Q e^{2 \sigma \psi} |p|^2 
+ \int_{B \times \{ T' \leq |t| \leq T \}} e^{2 \sigma \psi} (|\alpha|^2 + |\alpha_t|^2)
\nn
\\
& \cleq \int_Q e^{2 \sigma \psi} |p|^2 
+ e^{2 \sigma (c - \delta)} \int_{B \times \{ T' \leq |t| \leq T \}} |\alpha|^2 + |\alpha_t|^2.
\label{eq:tempLb99}
\end{align}
Using Lemma \ref{lemma:neartet} for $\alpha$ on $Q_+$, a corresponding result for $\alpha + \rho$ on $Q_-$, and
Lemma \ref{lemma:pweak} for $\rho$, 
along with (\ref{eq:alphade}), (\ref{eq:alphacc}),
 we have 
\begin{align*}
 \int_{B \times \{ T' \leq |t| \leq T \}} |\alpha|^2 + |\alpha_t|^2
~ & \cleq ~ \| (\Box + q) \alpha \|_{L^2(Q)}^2 + \|\alpha\|_{H^1(\Gamma)}^{2}
+  \|\alpha\|_{H^1(\Sigma_+)}^2 +  \|\pa_\nu \alpha\|_{L^2(\Sigma_+)}^2
\\
& \qquad 
 + \|\alpha + \rho \|_{H^1(\Sigma_-)}^2 +  \|\pa_\nu (\alpha + \rho)\|_{L^2(\Sigma_-)}^2
 + \|\rho \|_{H^1(\Sigma \cap \Gamma)}^2
 \\
 & \cleq \int_B |p|^2 + \|\alpha\|_{H^1(\Gamma)}^2 
 + \|\alpha\|_{H^1(\Sigma)}^2 +  \|\pa_\nu \alpha\|_{L^2(\Sigma)}^2
 + \|\rho \|_{H^1(\Sigma_-)}^2 +  \|\pa_\nu \rho\|_{L^2(\Sigma_-)}^2 
 \\
 & \qquad + \|\rho \|_{H^1(\Sigma \cap \Gamma) }^2
 \\
 & \cleq  \|\alpha\|_{H^1(\Gamma)}^2  
 +  \|\alpha\|_{H^1(\Sigma)}^2 +  \|\pa_\nu \alpha\|_{L^2(\Sigma)}^2 + 
 \|\rho \|_{H^1(\Sigma \cap \Gamma)}^2.
\end{align*}
Using this in (\ref{eq:tempLb99}) and noting $\psi|_\Gamma \geq c+ \delta$,  we obtain
\begin{align*}
\|(\Box + q) \beta \|_{0, \sigma, Q}^2 
& \cleq \int_Q e^{2 \sigma \psi} |p|^2 + e^{2\sigma(c-\delta)} \|\alpha\|_{H^1(\Gamma)}^2
\nn
\\
& \qquad + 
e^{2\sigma (c-\delta)} ( 
  \|\alpha\|_{H^1(\Sigma)}^2 
+  \|\pa_\nu \alpha\|_{L^2(\Sigma)}^2 + \|\rho \|_{H^1(\Sigma \cap \Gamma)}^2 )
\nn
\\
& \cleq 
\int_Q e^{2 \sigma \psi} |p|^2 + e^{-2 \sigma \delta}  \|\alpha\|_{1,\sigma, \Gamma}^2 
+ e^{k \sigma} ( \|\alpha\|_{H^1(\Sigma)}^2 
+  \|\pa_\nu \alpha\|_{L^2(\Sigma)}^2 + \|\rho \|_{H^1(\Sigma \cap \Gamma)}^2 )
\end{align*}
 for some $k>0$. 
Hence from (\ref{eq:betacarl}) and using $\beta = \chi \alpha$ for the terms on $\Sigma$
\beqn
\sigma \|\beta \|_{1,\sigma,Q}^2 \cleq 
\int_Q e^{2 \sigma \psi} |p|^2 + e^{-2 \sigma \delta}  \|\alpha\|_{1,\sigma, \Gamma}^2 
+ e^{k \sigma} \left  ( \|\alpha\|_{H^1(\Sigma)}^2 
+  \|\pa_\nu \alpha\|_{L^2(\Sigma)}^2 + \|\rho \|_{H^1(\Sigma \cap \Gamma)}^2  \right )
\label{eq:betagamma}
\eeqn
 From Lemma \ref{lemma:neartez} applied with $T'$ instead of $T$, 
noting $Q_+ \cap \{t \leq T'\} \subset Q$, $\Sigma_+ \cap \{t \leq T' \} \subset \Sigma_+$, and $\alpha=\beta$ on 
$Q_+ \cap\{t \leq T' \}$,
we have
\begin{align*}
\|\alpha\|_{1,\sigma,\Gamma}^2 
& \cleq \sigma \| \alpha \|_{1,\sigma, Q_{+,T'}}^2 + \| (\Box + q) \alpha \|_{0, \sigma, Q_{+,T'}}^2 + 
\|\alpha\|_{1, \sigma, \Sigma_{+,T'}}^2 + \|\pa_\nu \alpha \|_{0,\sigma, \Sigma_{+,T'}}^2
\nn
\\
& \cleq \sigma \| \beta \|_{1,\sigma, Q_+}^2 + \int_Q e^{2 \sigma \psi} |p|^2 +
 e^{k \sigma} \left (  \|\alpha\|_{H^1(\Sigma)}^2 + \| \pa_\nu \alpha \|_{L^2(\Sigma)}^2 \right ),
\end{align*}
for some $k>0$.
Using this and (\ref{eq:betagamma}) and noting $e^{-2\sigma \delta}$ is small compared to $1$ for large $\sigma$, 
we obtain
\begin{align}
\|\alpha\|_{1,\sigma,\Gamma}^2 
& \cleq 
\int_Q e^{2 \sigma \psi} |p|^2 + 
e^{k \sigma} \left ( \|\rho\|_{H^1(\Sigma \cap \Gamma)}^2 + \|\alpha\|_{H^1(\Sigma)}^2 
+ \|\pa_\nu \alpha \|_{L^2(\Sigma)}^2 \right ).
\label{eq:ptempaa}
\end{align}

From (\ref{eq:alphacc})  we have
\begin{align*}
e^{ \sigma \psi (y,z,z)} p(y,z)  & = e^{ \sigma \psi(y,z,z)} \frac{d}{dz} \int_{-\infty}^z p(y,s) \, ds
=  e^{ \sigma \psi(y,z,z)} \frac{d}{dz} (\alpha(y,z,z)) 
\\
& = 2 e^{ \sigma \psi (y,z,z)} ( \alpha_z + \alpha_t )(y,z,z),
\end{align*}
so (\ref{eq:ptempaa}) implies that for large enough $\sigma$
\begin{align}
\int_\Gamma e^{2 \sigma \psi} |p|^2
& \cleq 
\int_Q e^{2 \sigma \psi} |p|^2 +  e^{k \sigma} \left ( \|\rho\|_{H^1(\Sigma \cap \Gamma)}^2 + \|\alpha\|_{H^1(\Sigma)}^2 
+ \|\pa_\nu \alpha \|_{L^2(\Sigma)}^2 \right ).
\label{eq:ptemp000}
\end{align}
Using the definition of $h(\sigma)$ in Lemma \ref{lemma:integral} we have
\begin{align*}
\int_Q e^{2 \sigma \psi} |p|^2 
& = 
\int_B e^{2\sigma \psi(y,z,z)} |p(y,z)|^2 \int_{-T}^T e^{2 \sigma (\psi(y,z,t) - \psi(y,z,z))} \, dt \, dy \, dz 
\\
& \leq h(\sigma) \int_B e^{2\sigma \psi(y,z,z)} |p(y,z)|^2 \, dy \, dz,
\end{align*}
so (\ref{eq:ptemp000}) implies that
\begin{align*}
\int_B e^{2 \sigma \psi(y,z,z)} |p(y,z)|^2 \, dy \, dz \;
& \cleq \,
 h(\sigma)  \int_B e^{2\sigma \psi(y,z,z)} |p(y,z)|^2 \, dy \, dz  
\\
& \qquad 
+  e^{k \sigma} \left  ( \|\rho\|_{H^1(\Sigma \cap \Gamma)}^2 + \|\alpha\|_{H^1(\Sigma)}^2 
+ \|\pa_\nu \alpha \|_{L^2(\Sigma)}^2 \right ).
\end{align*}
Hence, from Lemma \ref{lemma:integral}, taking $\sigma$ large enough we obtain
\[
\int_B  e^{2 \sigma \psi(y,z,z)}  |p(y,z)|^2 \, dy \, dz 
\cleq
  e^{k \sigma} \left ( \|\rho\|_{H^1(\Sigma \cap \Gamma)}^2 + \|\alpha\|_{H^1(\Sigma)}^2 
+ \|\pa_\nu \alpha \|_{L^2(\Sigma)}^2 \right )
\]
with the constant dependent on $\|q\|_\infty$, $\|f\|_\infty$ and $T$. Fixing a large enough $\sigma$ (which also depends on $\|q\|_\infty$, $\|f\|_\infty$ and $T$) we get 
\[
\|p\|_{L^2(B)}^2 \cleq \|\rho\|_{H^1(\Sigma \cap \Gamma)}^2 + \|\alpha\|_{H^1(\Sigma)}^2
+  \|\pa_\nu \alpha \|_{L^2(\Sigma)}^2
\]
and the proof is complete.

%

\subsection{Proof of Proposition \ref{prop:pwell}}\label{subsec:pwell}
The only part of the proposition not proved in the proof of Theorem 1 in \cite{RU14} is the upper bound on $|u(x,t,\omega)|$. 
We use the notation in the proof of Theorem 1 in \cite{RU14}.  If $q$ is $C^{2N}$ then $(\Box +q)R_N$ is in $C^{N-1}$ and hence in 
$L^2_{loc}(\R, H^{N-1}(\R^n))$, so by Theorems 9.3.1, 9.3.2 and the remark after that in \cite{Ho76}, we conclude that
$R_N \in H^1(\R, H^{N-1}(\R^n))$ locally and, for any given $T$,
\[
\| R_N \|_{H^1((-\infty,T], H^{N-1}(\R^n))} \cleq \| (\Box + q) R_N \|_{L^2((-\infty,T], H^{N-1}(\R^n)) }
\cleq F( \|q\|_{C^{2N}})
\]
for some continuous function $F$, with the constant dependent on $\|q\|_{C^{N-1}}$ and $T$. Hence, if $N-1 >n/2$ then for any 
$t \in (-\infty,T]$ we have
\[
\| R_N(t, \cdot) \|_\infty \cleq \|R_N(t, \cdot) \|_{H^{N-1}(\R^n)} \cleq \| R_N \|_{H^1((-\infty,T], H^{N-1}(\R^n))}
\cleq F(\|q\|_{C^{2N}}),
\]
hence if $N-1>n/2$ then
\[
\|u(t, \cdot, \omega)\|_\infty \cleq F(\|q\|_{C^{2N}}), \qquad \forall t \leq T,
\]
so taking $N=2+\lfloor n/2 \rfloor$ or higher, we see that
\[
\|u(t, \cdot, \omega)\|_\infty \cleq F(\|q\|_{C^{n+4}}), \qquad \forall t \leq T.
\]
%



\section{Lemmas for the two plane wave sources problem}\label{sec:plemma}


We recall the following useful subsets of $\R^n \times \R$, 
\begin{align*}
Q := \Bb \times [-T,T], & \qquad \Sigma := S \times [-T,T], &  \hspace{-0.5in} \Gamma = Q \cap \{t=z\},
\\
Q_{\pm} := Q \cap \{ \pm(t-z) \geq 0 \},  & \qquad \Sigma_{\pm} = \Sigma \cap \{ \pm (t-z) \geq 0 \}, &
\\
Q_{+,\tau} = Q_+ \cap \{ t \leq \tau \}, & \qquad  \Sigma_{+,\tau} = \Sigma_+ \cap \{t \leq \tau \}, &
\end{align*}
for any $\tau \in \R$.

%
%

\subsection{Carleman weight and estimates for the plane waves problem}\label{subsec:pcarl}

\def\Omb{{\bar{\Omega}}}

There are some differences between definitions given in \cite{Ho76} and \cite{Ta96} for pseudoconvexity and
strong pseudo-convexity so we specify the definitions we plan to use. Suppose $P(x,D)$ is a differential operator 
with principal symbol $p(x,\xi)$ with real coefficients, over a region $\Omb$, and $\phi$ a smooth function on
$\Omb$ with $\nabla \phi \neq 0$ at each point of $\Omb$. We say the level surfaces of $\phi$ are pseudoconvex
w.r.t $P(x,D)$ on $\Omb$ if (1.3), (1.4) from \cite{Ta96} hold at every point of $\Omb$. We say the level surfaces
of $\phi$ are strongly pseudoconvex w.r.t $P(x,D)$ on $\Omb$ if (1.3)-(1.6) from \cite{Ta96} hold at every point 
of $\Omb$. We say the function $\phi$ is strongly pseudoconvex w.r.t $P(x,D)$ on $\Omb$ if (1.2) of \cite{Ta96} holds 
at every point of $\Omb$.

For second order operators $P(x,D)$ with real principal part, one may verify that the pseudoconvexity and strong 
pseudoconvexity conditions for level surfaces of $\phi$  are equivalent - see Theorem 1.8 on page 16 in \cite{Ta99}. 
Further, if the level surfaces of $\phi$ are strongly pseudoconvex w.r.t $P(x,D)$ in $\Omb$ then, from 
Theorem 8.6.3 in \cite{Ho76}, for large enough $\lambda$
\[
\psi = e^{\lambda \phi}
\]
is strongly pseudoconvex w.r.t $P(x,D)$ on $\Omb$. So to construct a strongly pseudoconvex weight $\psi$
for a second order operator with real coefficients, one just needs to construct a function $\phi$ whose level surfaces
are pseudoconvex w.r.t $P(x,D)$.

Our goal is to construct on $Q$ a function $\psi(y,z,t)$ strongly pseudoconvex w.r.t $\Box$ and
decreasing in $|t-z|$ for a fixed $y,z$.

%
%
\begin{lemma}\label{lemma:pweight}
Define
\[
\phi(y,z,t) := c (a-z)^2 + d |y|^2- (t-z)^2, \qquad (y,z) \in \R^n, ~ t \in \R
\]
and
\[
\psi(y,z,t) := e^{\lambda \phi(y,z,t)};
\]
then $\psi$ is strongly pseudoconvex w.r.t $\Box$ on the region $z \neq a$ if $c>4$, $d>c-1$, $\lambda > > 0$.
\end{lemma}
%
\begin{proof}
As explained at the beginning of subsection \ref{subsec:pcarl}, it is enough to prove that the level surfaces 
of $\phi$ are pseudoconvex on $z \neq a$ if $c>4$, $d>c-1$, and also that $\nabla_{y,z,t} \phi$ is non-zero at 
each point in the region $z \neq a$.

For convenience we take 
\[
\phi(y,z,t) = \frac{1}{2} (  c (a-z)^2 + d |y|^2- (t-z)^2 )
\]
and the principal symbol of $\Box$ to be
\[
p(y,z,t,\eta,\zeta,\tau) = \frac{1}{2} ( -\tau^2 + |\eta|^2 + \zeta^2), \qquad (\eta,\zeta,\tau) \in \R^{n-1} \times \R \times \R.
\]
We first note that $\nabla \phi $ is non-zero at every point in the region $z \neq a$ because 
$\phi_z = c (z  -a) \neq 0$ on the region $z \neq a$. Next, all $y,z,t$ derivatives of $p$ are zero, the mixed partials of $p$ and $\phi$ are zero (except for $\phi_{zt} = 1$), and
\begin{gather*}
p_\tau = - \tau, ~ \nabla_\eta p=  \eta, ~ p_\zeta =  \zeta,
\\ 
\phi_t = -(t-z), ~ \nabla_y \phi = d y, ~ \phi_z = c(z-a) -  (z-t),
\\
\phi_{tt} = -1,  ~~ (\phi_{y_i, y_j}) = d I_{n-1}, ~~\phi_{zz} =  c - 1, ~~ \phi_{zt}=1.
\end{gather*}
The condition (1.3) in \cite{Ta96}, in expanded form, is condition (8.4.5) in \cite{Ho76}, so
the level surfaces of $\phi$ are pseudoconvex w.r.t $\Box$ 
iff 
\begin{align*}
-(- \tau)^2 + d  \eta^T \eta+ (c-1) \zeta^2  + 2 \zeta (-\tau)>0
\end{align*}
whenever $(\eta, \zeta, \tau) \neq 0$ and
\[
\tau^2 = |\eta|^2 + \zeta^2, \qquad (-\tau) (-(t-z)) + \eta \cdot (dy) + \zeta \cdot (c(z-a) - (z-t)) =0.
\]
It will be enough to require that
\beqn
-\tau^2 + d |\eta|^2 + (c-1) \zeta^2 - 2\zeta \tau >0
\label{eq:carltemp}
\eeqn
whenever 
$(\eta, \zeta, \tau) \neq 0$ and $\tau^2 = |\eta|^2 + \zeta^2$.
Because of homogeneity, we can take $\tau = \pm1$ and $|\eta|^2 + \zeta^2 =1$, so it would be enough to require that
\[
f(\zeta) := -1 + d(1-\zeta^2) +(c-1) \zeta^2  \pm 2\zeta >0 \qquad \text{whenever} ~ \zeta \in [-1,1].
\]
Now 
\[
f(\zeta) = (c-d-1)\zeta^2 \pm 2\zeta + d-1
\]
is a downward opening parabola when $c-d-1<0$ so its minimum on $[-1,1]$ will be at the end points. Hence
the minimum of $f(\zeta)$ on $[-1,1]$ will be  $c-4$. So it will be enough to require that $c>4$ and $c<d+1$, that is $c>4$ and $d>c-1$.
 \end{proof}

 
 Next, we compute the limit of an integral associated with the Carleman weight we use in the proof of
 Proposition \ref{prop:planeprop}.
 \begin{lemma}\label{lemma:integral}
 If $\phi(y,z,t) = c (a-z)^2 + d |y|^2 - (t-z)^2$ with $a>1$, $c>0, d>0$, $\psi = e^{\lambda \phi}$ for some $\lambda>0$
 and
\[
h(\sigma) :=  \sup_{ (y,z) \in \Bb} \int_{-T}^T e^{2\sigma ( \psi (y,z,t) - \psi(y,z,z)) }\, dt
\]
for some $T>0$, then $\lim_{\sigma \to \infty} h(\sigma) =0.$
 \end{lemma}
 
 
 \begin{proof}
Since $\lambda >0$ and $\phi(y,z,z) \geq 0$, for any $ (y,z) \in \Bb$ we have
\begin{align*}
\psi(y,z,z) - \psi(y,z,t)   & = e^{ \lambda \phi(y,z,z)} - e^{\lambda \phi(y,z,t)}
 =  e^{\lambda \phi(y,z,z)} (1 -  e^{ - \lambda (\phi(y,z,z) -  \phi(y,z,t) ) } )
\\
& = e^{\lambda \phi(y,z,z)} ( 1- e^{ - \lambda (t-z)^2})
 \geq 1- e^{ - \lambda (t-z)^2}.
\end{align*}
Now, for $s \geq 0$, 
\[
1 - e^{-s} \geq \min( 1/2, s/2),
\]
hence
\[
2(\psi(y,z,t) - \psi(y,z,z)) \leq - \min(1, \lambda(t-z)^2 )
\]
so
\begin{align*}
\int_{-T}^T e^{2 \sigma (\psi(y,z,t)-\psi(y,z,z)} \, dt
& \leq \int_{-T}^T e^{ - \sigma \min( 1, \lambda (t-z)^2) } \, dt
 = \int_{-T-z}^{T-z} e^{ - \sigma \min( 1, \lambda t^2) } \, dt
\\
& \leq \int_{-T-1}^{T+1} e^{ - \sigma \min( 1, \lambda t^2) } \, dt.
\end{align*}
Hence, by the dominated convergence theorem
\[
0 \leq  \lim_{\sigma \to \infty} h(\sigma) \leq 
\lim_{\sigma \to \infty}  \int_{-T-1}^{T+1} e^{ - \sigma \min( 1, \lambda t^2) } \, dt =0.
\]
\end{proof}

%
%

\subsection{Energy estimates for the plane wave problem}

We derive three energy estimates needed in the proof of Theorem \ref{thm:twop1}. The first is an estimate 
for an exterior
problem, the second estimates the energy on $t=\pm T$ and the third estimates the energy on $t=z$. In deriving these estimates, we 
will use the following simple integration by parts results on $Q_+$ and other sets having a similar form: if $v$ is smooth in $Q_+$, 
then 
\[
\int_{Q_+} \pa_t v \,dx \,dt = \int_{\{t=T\}} v \,dS - \frac{1}{\sqrt{2}} \int_{\{t=z\}} v \,dS,
\]
and if $V$ is a smooth vector field on $Q_+$ with values in $\R^n$, then (with $\nabla$ denoting the gradient in $x$ variables) 
\[
\int_{Q_+} \nabla \cdot V \,dx \,dt = \int_{\Sigma_+} V \cdot \nu \,dS + \frac{1}{\sqrt{2}} \int_{\{t=z\}} V \cdot e_n \,dS.
\]

\begin{lemma}[Energy estimate for exterior problem]\label{lemma:pexterior}
Suppose $n>1$,  $T>1$, $p(x)$a smooth function on $\R^n$ with support in $\Bb$ and
$\alpha(y,z,t)$ a smooth function on $(\R^n \times \R) \cap \{ t \geq z \}$ 
with
\begin{align}
\Box \alpha =0, & \qquad \text{on} ~ (y,z,t) \in \R^n \times \R, ~ |(y,z)| \geq 1, ~ t \geq z,
\label{eq:poutde}
\\
\alpha(y,z,z) = \int_{-\infty}^0 p(y,z+s) \, ds, & \qquad \text{on}~ |(y,z)| \geq 1,
\label{eq:poutcc}
\\
\alpha(y,z,t) = 0, & \qquad \text{on} ~ \{ (y,z,t) : z \leq t < -1 \};
\label{eq:poutic}
\end{align}
then
\[
\| \pa_\nu \alpha \|_{L^2(\Sigma_+)} \cleq \| \alpha \|_{H^1(\Sigma_+)} + \|\alpha\|_{H^1(\Sigma \cap \Gamma)},
\]
with the constant dependent only on $T$.

\end{lemma}
%
%
\begin{proof}
The result follows from  standard estimates for the wave operator obtained using multiplier methods.
Define
\[
H_T = \{ (y,z,t): |(y,z)| \geq 1, ~ -T \leq t \leq T, ~ z \leq t \};
\]
then from domain of dependence arguments and (\ref{eq:poutcc}), (\ref{eq:poutic}), we can show that the
intersection of the support of $\alpha$ and $H_T$ is bounded and hence, on this set, $|x|$ is bounded above by a
constant dependent on $T$. 

\def\alphab{{\bar{\alpha}}}
We define the smooth function
\[
\alphab(y,z) := \alpha(y,z,z) =  \int_{-\infty}^0 p(y, z +s) \, ds = \int_{-\infty}^z p(y,s) \, ds,
\qquad (y,z) \in \R^n,
\]
and noting that $p$ is supported in $\Bb$, for $|(y,z)| \geq 1$ we have
\[
\alphab(y,z) = \alpha(y,z,z) =  \begin{cases} 
0, & \text{if} ~ |y| \geq 1,
\\
0, & \text{if} ~ |y| \leq 1 ~ \text{and} ~ z \leq 0,
\\
\int_\R p(y,s) \, ds, & |y| \leq 1 ~ \text{and} ~ z \geq 0,
\end{cases}
\] 

We have the identities
\begin{align*}
2 \alpha_t \Box \alpha &= ( \alpha_t^2 + |\nabla \alpha|^2)_t - 2 \nabla \cdot ( \alpha_t \nabla \alpha)
\\
2 (x \cdot \nabla \alpha) \, \Box \alpha & = 2 ( \alpha_t (x \cdot \nabla \alpha) )_t 
+ \nabla \cdot \left ( x ( |\nabla \alpha|^2 - \alpha_t^2) - 2 (x \cdot \nabla \alpha) \nabla \alpha \right )
+ n \alpha_t^2 - (n-2) |\nabla \alpha|^2.
\end{align*}
For any $\tau \in [-T,T]$, integrating the first identity over the region $H_T \cap \{ t \leq \tau \}$ and noting 
that $\alpha$ is compactly supported for
each fixed $t$, $\alpha(y,z,z) = \alphab(y,z)$ and $(\alpha_t + \alpha_z)(y,z,z)=0$ on $ \{t=z\} \cap H_T$,
we have
\begin{align}
\int_{H_T \cap \{t=\tau \}} \alpha_t^2  &+ |\nabla \alpha |^2 \, dx
 = - 2 \int_{\Sigma_+ \cap \{t \leq \tau\}} \alpha_r \, \alpha_t \, dS + 
\int_{H_T \cap \{t=z\}} (\alpha_t^2 + |\nabla \alpha|^2 + 2 \alpha_t \alpha_z)(y,z,z) \, dy \, dz
\nonumber
\\
& = - 2 \int_{\Sigma_+ \cap \{t \leq \tau\}} \alpha_r \, \alpha_t \, dS + 
\int_{H_T \cap \{t=z\}} |\nabla_y \alpha(y,z,z)|^2  + (\alpha_t + \alpha_z)(y,z,z)^2 \, dy \, dz
\nonumber
\\
& \leq  \ep \int_{\Sigma_+} \alpha_r^2 \, dS+ \frac{1}{\ep} \int_{\Sigma_+} \alpha_t^2 \, dS
+ \int_{H_T \cap \{t=z\}} |\nabla_y \alphab(y,z)|^2  \, dy \, dz
\label{eq:ep1}
\end{align}
for all $\ep>0$.
Integrating the second relation over $H_T$  we obtain
\begin{align}
\int_{\Sigma_+} 2 \alpha_r^2 +\alpha_t^2 - |\nabla \alpha|^2 \, dS
&= \int_{H_T} ( (n-2) |\nabla \alpha|^2 - n \alpha_t^2 ) \, dx \, dt - 2 \int_{|x| \geq 1, t=T} \alpha_t \, (x \cdot \nabla \alpha) \, dx
\nonumber
\\
& \qquad + \int_{H_T \cap \{t=z\}} 2 (\alpha_t + \alpha_z) (x \cdot \nabla \alpha) - z |\nabla_y \alpha|^2 
+ z (\alpha_t ^2 - \alpha_z^2)
\nonumber
\\
& =  \int_{H_T} ( (n-2) |\nabla \alpha|^2 - n \alpha_t^2 ) \, dx \, dt - 2 \int_{|x| \geq 1, t=T} \alpha_t \, (x \cdot \nabla \alpha) \, dx
\nn
\\
& \qquad - \int_{H_T \cap \{t=z\}} z |\nabla_y \alphab(y,z)|^2 \, dy \, dz.
\label{eq:alphaside}
\end{align}
Hence using (\ref{eq:fangular}) and (\ref{eq:ep1}) we obtain from (\ref{eq:alphaside}) that
\begin{align*}
\int_{\Sigma_+} \alpha_r^2  \, dS
& \cleq \int_{\Sigma_+} \sum_{i<j} (\Omega_{ij} \alpha)^2 \, dS 
+ 
 \int_{H_T} \alpha_t^2 + |\nabla \alpha|^2 \, dx \, dt
+ \int_{H_T \cap \{t=T\}} \alpha_t^2 + |\nabla \alpha|^2 \, dx
\\
& \qquad + \int_{H_T \cap \{t=z\}}  |\nabla_y \alphab(y,z)|^2 \, dy \, dz
\\
& = \int_{\Sigma_+}  \sum_{i<j} (\Omega_{ij} \alpha)^2 \, dS  + \int_{-T}^T \int_{H_T \cap \{ t= \tau \}} \alpha_t^2 + |\nabla \alpha|^2 \, dx \, d \tau + \int_{H_T \cap \{t=T\}} \alpha_t^2 + |\nabla \alpha|^2 \, dx
\\
& \qquad + \int_{H_T \cap \{t=z\}}  |\nabla_y \alphab(y,z)|^2 \, dy \, dz
\\
& \leq \left (1 + \frac{2 T}{\ep} \right  )  \int_{\Sigma_+} \alpha_t^2 + \sum_{i<j} (\Omega_{ij} \alpha)^2 \, dS 
+ (2T+1)\ep \int_{\Sigma_+} \alpha_r^2 \, dS
\\
& \qquad + \int_{H_T \cap \{t=z\}}  |\nabla_y \alphab(y,z)|^2 \, dy \, dz
\end{align*}
So, choosing $\ep$ small enough, we obtain
\[
\int_{\Sigma_+} \alpha_r^2  \, dS
\cleq \| \alpha \|_{H^1(\Sigma_+)}^2 + \int_{H_T \cap \{t=z\}}  |\nabla_y \alphab(y,z)|^2 \, dy \, dz.
\]

Now, $\alphab(y,z)=0$ for $|y| \geq 1$ and when $z \leq 0$. Further,
\[
\alphab(y,z) = \alphab(y, \sqrt{1-|y|^2}), \qquad \text{when} ~ z \geq 0, ~ |y| \leq 1
\]
so
\begin{align*}
\int_{H_T \cap \{t=z\}}  |\nabla_y \alphab(y,z)|^2 \, dy \, dz ~
& 
\cleq
\int_{|y| \leq 1} |\nabla_y ( \alphab(y, \sqrt{1-|y|^2}) |^2 \, dy
\\
& =\int_{|y| \leq 1}  | (\nabla_y \alphab - y/\sqrt{1-|y|^2} \alphab_z) (y, \sqrt{1-|y|^2}) |^2 \, dy
\\
& = \int_{|y| \leq 1} \frac{1}{\sqrt{1-|y|^2}} | ( ( \sqrt{1-|y|^2} \nabla_y - y \pa_z) \alphab )(y, \sqrt{1-|y|^2})|^2
\, dy
\\
& \cleq \int_S | ( (z \nabla_y - y \pa_z) \alphab )(y,z)|^2 \, dS_{y,z}
\\
& \cleq \|\alphab\|_{H^1(S)}^2
\\
& \cleq  \|\alpha\|_{H^1(\Sigma \cap \{t=z\})}^2
\end{align*}
and the proof is complete.
\end{proof}


Next we estimate the energy near $t= T$ by the energy on $t=z$.
\begin{lemma}[Energy estimate near $t= T$]\label{lemma:neartet}
If $1<\tau\leq T$, $q$ is a smooth function on $\Bb$ and $\alpha(x,t)$ is a smooth function on $Q_+$ then
\begin{align*}
\int_B  (|\nabla_{x,t}  & \alpha|^2 + |\alpha|^2)(x, \tau) \, dx
  \cleq \|\alpha\|_{H^1(\Gamma)}^2 + \|(\Box +q) \alpha\|_{L^2(Q_+)}^2 + \|\alpha\|_{H^1(\Sigma_+)}^2
+ \| \pa_\nu \alpha \|_{L^2( \Sigma_+)}^2,
 \end{align*}
 with the constant dependent only on $\|q\|_\infty$ and $T$.
\end{lemma}
%

\begin{figure}[h]
\begin{center}
\epsfig{file=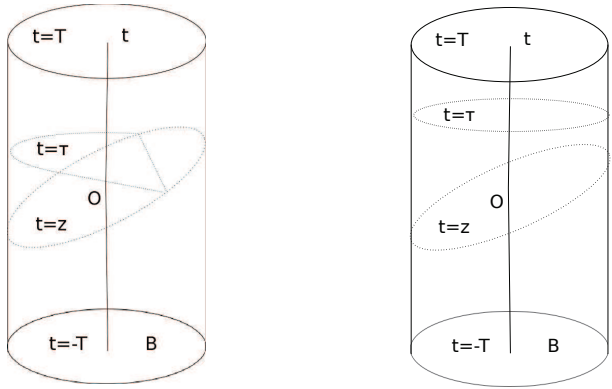, height=2in}
\end{center}
\caption{The $t=\tau$ section of $Q$}
\label{fig:penergy}
\end{figure}

\begin{proof}

Below $\L = \Box + q$.
For any $\tau \in [-1,T]$, the plane $t=\tau$ cuts $t=z$ inside $Q$ when $\tau \leq 1$ and does not cut $t=z$ 
inside $Q$ when $\tau>1$ - see Figure \ref{fig:penergy}. We define two energies associated with top and 
bottom surfaces of the boundary of $Q_{+,\tau}$.
\vspace{-0.2in}
\begin{itemize}
\item If $-1 \leq \tau \leq 1$ then define
\begin{align*}
E(\tau) &:= \int_{B \cap \{ z \leq \tau \}} ( \alpha^2 + \alpha_t^2 + |\nabla \alpha|^2)(y,z,\tau) \, dz \, dy;
\\
J(\tau) &:=  \int_{B \cap \{ z \leq \tau \}}  ( \alpha^2 + (\alpha_t+ \alpha_z)^2 + |\nabla_y \alpha|^2)(y,z,z) \, dz \, dy;
\end{align*}
\item If $\tau>1$ then define
\begin{align*}
E(\tau)  &:=  \int_B ( \alpha^2 + \alpha_t^2 + |\nabla \alpha|^2)(y,z,\tau) \, dz  \, dy;
\\
J(\tau) & := J(1).
\end{align*}
\end{itemize}
\vspace{-0.2in}

For any $\tau \in [-1,T]$, integrating the relation
\[
2\alpha_t (\Box \alpha + \alpha) = ( \alpha_t^2 + |\nabla \alpha|^2 + \alpha^2)_t - 
2 \nabla \cdot (\alpha_t \nabla \alpha),
\]
over the region $Q_{+,\tau}$  we obtain
\begin{align*}
E(\tau) & = J(\tau)  + 2\int_{ Q _{+,\tau}} \alpha_t (\Box \alpha + \alpha )
 + 2  \int_{\Sigma_{+,\tau}} \alpha_t \pa_{\nu} \alpha\, dS
 \\
 & = J(\tau) + 2\int_{ Q_{+,\tau}} \alpha_t (\L \alpha + (1-q) \alpha )
 + 2  \int_{\Sigma_{+,\tau}}  \alpha_t \pa_{\nu} \alpha\, dS
 \\
 & \cleq J(1)  + \int_{Q_+} |\L \alpha|^2
 + \int_{\Sigma_+} |\nabla_{x,t} \alpha|^2 + | \alpha |^2  \, dS   + \int_0^\tau E(t) \, dt.
\end{align*}
with the constant dependent on $\|q\|_\infty$. Hence, by Gronwall's inequality, for all $\tau \in [-1,T]$, we have
\[
E(\tau) \cleq J(1) + \int_{Q_+} |\L \alpha|^2
 + \int_{\Sigma_+} |\nabla_{x,t} \alpha|^2 + | \alpha |^2  \, dS
\]
with the constant dependent on $T$ and $\|q||_\infty$. 
\end{proof}

%

Next we estimate the weighted energy on $t=z$.
\begin{lemma}[Energy estimate near $t=z$]\label{lemma:neartez}
If $T>1$, $q$ a smooth function on $\Bb$ and $\alpha, \psi$ are smooth functions on $Q_+$ then 
\[
\|\alpha\|_{1,\sigma,\Gamma}^2
\, \cleq \,
\sigma  \|\alpha\|_{1,\sigma, Q_+}^2
+ \| (\Box + q) \alpha \|_{0, \sigma, Q_+}^2
 + \| \alpha\|_{1, \sigma, \Sigma_+}^2 + \| \partial_{\nu} \alpha\|_{0, \sigma, \Sigma_+}^2
\]
for all $\sigma>1$ large enough, with the constant dependent only on 
$\|q\|_\infty$, $\|\psi\|_{C^2(Q_+)}$ and $T$.
\end{lemma}


\begin{proof}

Below $\L = \Box + q$.
Define $\mu := e^{\sigma \psi} \alpha$.
For any $\tau \in [1,T]$, define the energy on the plane $t=\tau$ as (see Figure \ref{fig:penergy})
\[
E(\tau)  :=  \int_B ( \sigma^2 \mu^2 + \mu_t^2 + |\nabla \mu|^2)(y,z,\tau) \, dz  \, dy
\]
and the energy on the plane $t=z$ as
\[
J := \int_{B}  ( \sigma^2 \mu^2 + (\mu_t+ \mu_z)^2 + |\nabla_y \mu|^2)(y,z,z) \, dz \, dy.
\]
For any $\tau \in [1,T]$, integrating the relation
\beqn
2\mu_t (\Box \mu + \sigma^2 \mu) = ( \mu_t^2 + |\nabla \mu|^2 + \sigma^2 \mu^2)_t - 
2 \nabla \cdot (\mu_t \nabla \mu),
\label{eq:pmuiden}
\eeqn
over the region $Q_{+,\tau}$  and using $2 \sigma^2 \mu \mu_t \leq \sigma ( \mu_t^2 + \sigma^2  \mu^2)$),
we obtain
\begin{align*}
J & = E(\tau) -  2\int_{ Q_{+,\tau}} \mu_t (\L \mu + (\sigma^2-q) \mu )
 - 2  \int_{\Sigma_{+,\tau}}  \mu_t \pa_{\nu} \mu\, dS
 \\
 & \cleq E(\tau) + \sigma \int_{Q_+} |\nabla_{x,t} \mu|^2 + \sigma^2 \mu^2 + \int_{Q_+} |\mu_t \L \mu|  + 
 \int_{\Sigma_+} | \mu_t \pa_{\nu} \mu | \, dS.
 \end{align*}
 Integrating this over $\tau \in [1,T]$ and noting that 
 \[
 \int_1^T E(\tau) \, d \tau \cleq \int_{Q_+} |\nabla_{x,t} \mu|^2 + \sigma^2 \mu^2
 \]
 we obtain
 \beqn
 J \cleq \, \sigma \int_{Q_+} |\nabla_{x,t} \mu|^2 + \sigma^2 \mu^2 + \int_{Q_+} | \mu_t \L \mu |  + 
 \int_{\Sigma_+}  |\mu_t \pa_{\nu} \mu |\, dS.
 \label{eq:pJ1}
 \eeqn
 with the constant dependent only on $T$ and $\|q\|_\infty$.
  
Now $\mu = \alpha e^{\sigma \psi}$ so
\begin{align}
|\mu|  \cleq e^{ \sigma \psi} |\alpha|, \qquad
|\nabla_{x,t} \mu|  \cleq e^{\sigma \psi} ( |\nabla_{x,t} \alpha| + \sigma |\alpha|).
\label{eq:pmu2}
\end{align}
Further, since $\alpha$ is smooth in the region $t \geq z$, on $Q_+$ we have
\begin{align*}
\L \mu & =(\Box +q) \mu  = e^{\sigma \psi} \left ( \L \alpha
+ 2 \sigma (\psi_t \alpha_{t} -  \nabla \psi \cdot \nabla \alpha)
+ \sigma \alpha \Box \psi + \sigma^2 (\psi_t^2 - |\nabla \psi|^2) \alpha \right )
\end{align*}
which implies
\beqn
|\L \mu| \cleq e^{\sigma \psi} ( |\L \alpha| + \sigma |\nabla_{x,t} \alpha| + \sigma^2 |\alpha| ).
\label{eq:pmu3}
\eeqn
Hence using (\ref{eq:pmu2}), (\ref{eq:pmu3}) and noting 
\[
2 \sigma^2 |\nabla_{x,t} \alpha| \, |\alpha| \leq \sigma |\nabla_{x,t} \alpha|^2 + \sigma^3 \alpha^2,
\]
(\ref{eq:pJ1}) implies
\begin{align}
J & \cleq \| \L \alpha \|_{0,\sigma, Q_+}^2 + \sigma \|\alpha\|_{1, \sigma, Q_+}^2 + \|\alpha\|_{1, \sigma, \Sigma_+}^2 + \| \partial_{\nu} \alpha\|_{0, \sigma, \Sigma_+}^2
\label{eq:pJcarl}
\end{align}
with the constant dependent only on $T$ and $\|q\|_\infty$.

Now $\alpha = e^{-\sigma \psi} \mu$ so $|\alpha| \leq e^{-\sigma \psi} |\mu|$ and
\begin{align*}
|\alpha_z + \alpha_t| & \cleq e^{-\sigma \psi} ( |\mu_z + \mu_t| + \sigma |\mu|),
\\
|\nabla_y \alpha| & \cleq e^{-\sigma \psi} ( |\nabla_y \mu| + \sigma |\mu|)
\end{align*}
hence
\begin{align*}
|\alpha_z + \alpha_t|^2 + |\nabla_y \alpha|^2 +
\sigma^2 \alpha^2& \cleq e^{-2\sigma \psi} ( |\mu_z + \mu_t|^2 + |\nabla_y \mu|^2 + \sigma^2 |\mu|^2),
\end{align*}
so the lemma follows from (\ref{eq:pJcarl}) and the definition of $J$.
\end{proof}

%
%

\subsection{The construction of $\rho(x,t)$}

\begin{lemma}[The estimates for $\rho(x,t)$]\label{lemma:pweak}
If $T>1$, $q \in C_c^\infty(\R^n)$ and $\rhob(y)$ is a compactly supported smooth function on $\R^{n-1}$ then the 
characteristic boundary value problem
\begin{align}
(\Box + q) \rho =0, & \qquad \text{on} ~ Q_-
\label{eq:pwde}
\\
\rho(y,z,z) = \rhob(y), & \qquad \text{on}~ \Gamma
\label{eq:pwcc}
\\
\rho(y,z,t) = \rhob(y), & \qquad \text{on} ~ \Sigma_-
\label{eq:pwbc}
\end{align}
has a unique solution in $H^1(Q_-)$ with  $\pa_\nu \rho \in L^2(\Sigma_-)$, $\rho(\cdot,\tau) \in H^1(B)$,
$\rho_t(\cdot, \tau) \in L^2(B)$
 and
\[
\|\rho\|_{H^1(\Gamma)} + \|\rho(\cdot,\tau)\|_{H^1(B)} + \| \rho_t(\cdot,\tau)\|_{L^2(B)}
+ \| \pa_{\nu} \rho \|_{ L^2(\Sigma_-) } 
+ \|\rho\|_{H^1(\Sigma_-)}
\cleq
\|\rho\|_{H^1(\Sigma \cap \Gamma)},
\]
for all $\tau \in [-T,-1]$,
with the constant dependent only on $\|q\|_\infty$ and $T$.
Further
\[
(\Box + q) (\rho(y,z,t) H(z-t) ) = 0 \qquad \text{on} ~~ Q.
\]
\end{lemma}
%
%
\begin{proof}

Below $\L = \Box + q$.
Since $\rhob(y)$ is also a smooth function on $\R^n$ independent of $z$, we redefine 
$\rhob$ to be a smooth
compactly supported function on $\R^n$ which agrees with the old $\rhob$ on a neighborhood of $\Bb$. This 
redefinition does not change the lemma and avoids introducing a new symbol.

Arguing as one would to prove Proposition \ref{prop:pwell} (see the proof of Theorem 1 in \cite{RU14}), the characteristic IVP
\begin{align}
(\Box +q)\rhobv =0, & \qquad (y,z,t) \in \R^n \times \R, ~ t \geq z,
\label{eq:rhobvde}
\\
\rhobv(y,z,z) = \rhob(y,z), & \qquad (y,z) \in \R^n,
\label{eq:rhobvcc}
\\
\rhobv(y,z,t) =0, & \qquad z \leq t <<0,
\label{eq:rhobvic}
\end{align}
is well posed and has a smooth solution.
On $\Sigma$, define the function
\[
\ft(y,z,t) = \begin{cases} \rhobv(y,z,t), & \text{on} ~ \Sigma_+, \\ \rhob(y,z), & \text{on} ~ \Sigma_-; \end{cases}
\] 
 $\ft$ is in $H^1(\Sigma)$ because $\rhobv=\rhob$ on $\Gamma$. 
 Hence, by standard theory (see Theorem 3.1 in \cite{BY17}), the backward IBVP
\begin{align}
(\Box + q)\rhot =0, & \qquad \text{on} ~ Q,
\label{eq:rhotde}
\\
\rhot = \rhobv, & \qquad \text{on} ~ Q \cap \{ t \geq 1 \},
\label{eq:rhotic}
\\
\rhot = \ft, & \qquad \text{on} ~ \Sigma,
\label{eq:rhotbc}
\end{align}
has a unique solution in $H^1(Q)$ with $\pa_\nu \rhot  \in L^2(\Sigma)$ and $\rhot(\cdot,\tau) \in H^1(B)$,
 $\rhot_t(\cdot, \tau) \in L^2(B)$ for all $\tau \in [-T,-1]$.
Also, from domain of dependence arguments one can see that $\rhot=\rhobv$ on $Q_+$, and in particular $\rhot = \rhob$ on 
$\Gamma$. Let $\rho$ be the restriction of $\rhot$ to $Q_-$; this is the desired solution. 


To prove uniqueness, we need to show that if $\rhob(y)=0$ and $\rho \in H^1(Q_-)$ is a
solution of (\ref{eq:pwde})-(\ref{eq:pwbc}), then $\rho=0$. Given a smooth function $\psi$ on $Q_-$ 
which is supported in the interior of $Q_-$, 
the IBVP
\begin{align*}
(\Box + q)\phi = \psi, & \qquad (x,t) \in B \times (-T,\infty)
\\
\phi(\cdot, -T) =0, ~ \pa_t \phi(\cdot,-T)=0, & \qquad \text{on} ~B
\\
\pa_\nu \phi = 0, & \qquad \text{on}  ~S \times [-T,\infty)
\end{align*}
has a solution which is smooth on $\Bb \times [-T,\infty)$ (from Theorem 5.1 in Chapter IV of \cite{La85} and its application to $t$
derivatives of $\phi$).
Hence, using the definition of a weak solution of (\ref{eq:pwde}) - (\ref{eq:pwbc}) with $\rhob=0$, we have
\[
\int_{Q-} \rho \, (\Box + q) \phi =0;
\]
note that there is no contribution from the boundary of $Q_-$, not even from $\Gamma$,
 because the boundary terms on $\Gamma$ 
involve $\rho$ or the first order derivatives of $\rho$ in directions tangential to $\Gamma$. Hence
\[
\int_{Q_-} \rho \, \psi =0
\]
for every smooth function $\psi$ on $Q_-$ which is supported in the interior of $Q_-$. Hence $\rho=0$ on $Q_-$.


We next show that $(\Box + q) ( \rho(y,z,t) H(z-t) )=0$ on $Q$. Let $\phi(x,t)$ be a smooth 
function on $Q$ with support in the interior of $Q$. Noting that $\rhot$ is smooth on $Q_+$ and
$(\Box + q)\rhot=0$ on $Q_+$, from the construction of $\rhot$ we know that
\begin{align*}
0 & = \la \L  \rhot, \phi \ra = \la \rhot, \L \phi \ra
= \int_{Q_+} \rhobv \, \L \phi + \int_{Q_-} \rho \, \L \phi
\\
& = \int_{Q_+} \phi \, \L \rhobv + (\rhobv \phi_t - \rhobv_t \phi)_t + 
\nabla_x \cdot (\phi \nabla_x \rhobv -\rhobv \nabla_x \phi )
 + \int_{Q_-} \rho \, \L \phi
 \\
 & =  \int_{Q_-} \rho \, \L \phi 
 + \frac{1}{\sqrt{2}} \int_\Gamma \phi (\rhobv_t +\rhobv_z)  - \rhobv (\phi_t + \phi_z) \, dS
 \\
 & =  \int_{Q_-} \rho \, \L \phi 
 + \int_{\Bb} \phi(y,z,z) \frac{d}{dz} (\rhobv(y,z,z)) - \rhobv(y,z,z) \frac{d}{dz} (\phi(y,z,z))
 \, dy \, dz
 \\
  & =  \int_{Q_-} \rho \, \L \phi + 2 \int_{\Bb} \phi(y,z,z) \frac{d}{dz} ( \rhobv(y,z,z) ) \, dy \, dz
  \\
  & = \int_{Q_-} \rho \, \L \phi + 2 \int_{\Bb} \phi(y,z,z) \frac{d}{dz}( \rhob(y))  \, dy \, dz
  \\
  & =  \int_{Q_-} \rho \, \L \phi.
\end{align*}
Hence $\L ( \rho(y,z,t) H(z-t) ) =0$ on $Q$.


We now obtain the estimate in the lemma. We construct a sequence of functions $\rho_k \in H^2(Q_-)$ such
that $\rho_k =\rho$ on $\Gamma$,  $\rho_k( \cdot,\tau) \to \rho(\cdot, \tau)$ in $H^1(B)$, 
$\pa_t \rho_k(\cdot, \tau) \to \rho_t(\cdot,\tau)$ in $L^2(B)$ for all $\tau \in [-1,-T]$ and $\rho_k \to \rho$ in $H^1(\Sigma_-)$,
$\pa_\nu \rho_k \to \pa_\nu \rho $ in $L^2(\Sigma_-)$. Using multiplier methods and energy estimates as in the proof of 
Lemma 3.6 in \cite{BY17}, one 
can show that
\[
\|\rho_k(\cdot,\tau)\|_{H^1(B)} + \|\pa_t \rho_k(\cdot,\tau)\|_{L^2(B)} 
+ \| \pa_\nu \rho_k \|_{L^2(\Sigma_-)}
\cleq \| \rho_k\|_{H^1(\Gamma)} + \|\rho_k\|_{H^1(\Sigma_-)},
\]
for all $\tau \in [-T,-1]$.
Hence, letting $k \to \infty$, we obtain
\[
\|\rho(\cdot,\tau)\|_{H^1(B)} + \|\pa_t \rho(\cdot,\tau)\|_{L^2(B)} 
+ \| \pa_\nu \rho \|_{L^2(\Sigma_-)}
\cleq \| \rho\|_{H^1(\Gamma)} + \|\rho\|_{H^1(\Sigma_-)}.
\]
So the estimate in the lemma will follow if we can show that 
\beqn
\|\rho\|_{H^1(\Gamma)} + \|\rho\|_{H^1(\Sigma_-)}
\cleq \|\rho\|_{H^1(\Sigma \cap \Gamma)}.
\label{eq:rhorhorho}
\eeqn
Now
\[
\| \rho \|_{H^1(\Sigma \cap \Gamma)}^2
=
\sqrt{2}
\int_S |\rhob(y)|^2 + |z \nabla_y \rhob (y)|^2 + \sum_{i \neq j} |( y_i \pa_{y_j} - y_j \pa_{y_i}) \rhob(y)|^2 \, dS_{y,z}.
\]
On the other hand, on $\Gamma$, $\rho(y,z,z) = \rhob(y)$ so
\begin{align*}
\| \rho \|_{H^1(\Gamma)}^2 
& =\sqrt{2} \int_B |\rhob(y)|^2 + |\nabla_y \rhob (y)|^2 \, dy \, dz
 \cleq  \int_{|y| \leq 1} \sqrt{1-|y|^2} ( |\rhob(y)|^2 + |\nabla_y \rhob (y)|^2 ) \, dy
\\
&  \cleq \int_S  z^2 (  |\rhob(y)|^2 + |\nabla_y \rhob (y)|^2 ) \, dS_{y,z}
\\
& \cleq \| \rho \|_{H^1(\Sigma \cap \Gamma)}^2.
\end{align*}
Further, on $\Sigma_-$ the tangential derivatives of $\rho$ are derivatives
in the directions $\pa_t$, $ z \nabla_y - y \pa_z $ and
$ y_i \pa_{y_j} - y_j \pa_{y_i}$ of the function $\rhob(y)$ which is a smooth extension of the restriction of $\rho$ to $\Sigma_-$.
Hence
\begin{align*}
\|\rho\|_{H^1(\Sigma_-)}^2
&
= \int_{\Sigma_-} |\rhob(y)|^2 + | z \nabla_y \rhob(y) |^2 
+ \sum_{i \neq j} |( y_i \pa_{y_j} - y_j \pa_{y_i}) \rhob(y)|^2 \, dS_{y,z,t}.
\\
& \cleq \int_S |\rhob(y)|^2 + | z \nabla_y \rhob(y) |^2 
+ \sum_{i \neq j} |( y_i \pa_{y_j} - y_j \pa_{y_i}) \rhob(y)|^2 \, dS_{y,z}
\\
& \cleq \| \rho \|_{H^1(\Sigma \cap \Gamma)}^2.
\end{align*}
So we have proved (\ref{eq:rhorhorho}).

It remains to construct the approximating sequence $\rho_k$. From (\ref{eq:rhobvde}), one has
\[
(\pa_t + \pa_z) ( \rhobv_t - \rhobv_z)(y,z,z) = \Delta_y \rhob(y,z) - q(y,z) \rhob(y,z),
\qquad (y,z) \in \Bb
\]
which implies
\[
\frac{d}{dz} (  ( \rhobv_t - \rhobv_z)(y,z,z) ) = \Delta_y \rhob(y,z) - q(y,z) \rhob(y,z),
\qquad (y,z) \in \Bb
\]
which combined with
\[
(\rhobv_t + \rhobv_z)(y,z,z) = \frac{d}{dz} ( \rhob(y,z) ) = 0, \qquad (y,z) \in \Bb
\]
determines
\[
\rhobv_t(y,z,z) = \frac{1}{2} \int_{-\infty}^z  \Delta_y \rhob(y,s) - q(y,z) \rhob(y,s) \, ds, \qquad (y,z) \in \Bb.
\]
Let us define the smooth function
\[
g(y,z) := \frac{1}{2} \int_{-\infty}^z  \Delta_y \rhob(y,s) - q(y,z) \rhob(y,s) \, ds, \qquad (y,z) \in \R^n.
\]
Construct a smooth function $\chi$ on $(-\infty,0]$ with support in $[-1,0]$ such that 
\[
\chi(0)=0, ~~ \chi'(0)=1
\]
and define
\[
\ft_k(y,z,t) = \begin{cases} 
\rhobv(y,z,t), & \qquad (y,z,t) \in \Sigma_+, \\
\rhob(y) + k^{-1} \chi(k(t-z)) g(y,z), & \qquad (y,z,t) \in \Sigma_-. 
\end{cases}
\]
We note that $\ft_k \in C^1(\Sigma)$ because $\ft_k(y,z,z+) = \ft_k(y,z,z-)$ and
$\pa_t \ft_k (y,z,z+) = \pa_t \ft(y,z,z-)$ for $(y,z) \in S$. 
Further $\ft_k \to \ft$ in $H^1(\Sigma)$ because
\begin{align*}
\| k^{-1} \chi(k(t-z) g(y,z) \|_{H^1(\Sigma_-)}^2
& \cleq \frac{1}{k^2} \int_{\Sigma_-} \chi^2(k(t-z)) \, dS + \int_{\Sigma_-} \chi'(k(t-z))^2 \, dS
\\
& \cleq \frac{1}{k^2} + \int_{-1}^0 |\chi'(ks)|^2 \, ds
\\
& \cleq \frac{1}{k^2} + \int_{-1/k}^0 \, ds ~~ \to ~~ 0, \qquad \text{as $k \to \infty$}.
\end{align*}

Let $\rhot_k$ be the solution of the IBVP (\ref{eq:rhotde})-(\ref{eq:rhotbc}) except with $\ft$ replaced by $\ft_k$. 
Since
$\ft_k$ is in $C^1$ and $\rhot_k = \rhot$ on $ t \geq z$, by applying Theorem 3.1 in \cite{BY17} to $\rhot_k$ and 
$\pa_t \rhot_k$ one can show that $\rhot_k \in H^2(Q)$ and we have
\[
\|(\rhot - \rhot_k)(\cdot,\tau)\|_{H^1(B)} + \| \pa_t (\rhot - \rhot_k)(\cdot,\tau)\|_{L^2(B)}
+ \| \pa_\nu (\rhot - \rho_k) \|_{ L^2(\Sigma) }  \cleq 
\| \ft_k - \ft \|_{H^1(\Sigma)},
\]
for all $\tau \in [-T,-1]$.
So if we take $\rho_k$ to be the restriction of $\rhot_k$ to $Q_-$ then we have constructed the desired $\rho_k$. Note that
on $\Gamma$ we have $\rho_k = \rhot_k = \rhobv$ because $\rhot_k = \rhobv$ on $Q_+$ by a domain of dependence argument.
\end{proof}




\section{Proofs for the spherical and point source problem}


%
Our functions will be defined mostly over the region above $t=-|x|$ and we avoid points where $x=0$ so we define
\[
K = \{ (x,t) \in \R^3 \times \R : -|x|<t, ~ x \neq 0 \}.
\]
The following proposition will be crucial in the proof of Theorem \ref{thm:pointspherical}.
%
\begin{prop}[Main proposition for spherical and point source problem]\label{prop:sprop}
Suppose $p(x), q(x)$ are smooth functions on $\R^3$, supported in $\Bb$ and zero in neighborhood of the origin. 
Let $f(x,t)$ be a bounded function on $K$ and $\alpha(x,t)$ a continuous function on $K$ with 
$t$ sections of $\alpha$ compactly supported, $\alpha$ smooth on the subregions $t \geq |x|$, $t \leq |x|$ with
\begin{align}
(\Box + q) \alpha(x,t) = p(x) f(x,t), & \qquad (x,t) \in K, \qquad (\text{as distributions})
\label{eq:salphade}
\\
\alpha(x,|x|) = \int_0^1 p(sx)  \,ds, & \qquad x \neq 0,
\label{eq:salphacc}
\end{align}
and
$\alpha|_{S \times (-1,3)} =0$, $\pa_r \alpha|_{S \times (-1,3)}=0$; then $p=0$.
\end{prop}

%
%

\subsection{Proof of Theorem \ref{thm:pointspherical}}
\begin{figure}[h]
\begin{center}
\epsfig{file=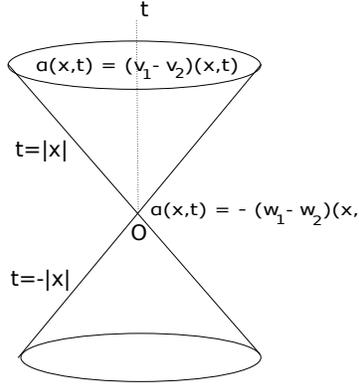, height=2.0in}
\end{center}
\caption{Definition of $\alpha$}
\label{fig:vtilde}
\end{figure}

Suppose $q_i(x)$, $i=1,2$, are smooth functions on $\R^3$ with support in $\Bb$ and zero in a neighborhood of the
origin. Let $v_i$ and $w_i$, $i=1,2$, be the functions, corresponding to $q=q_i$, whose existence and uniqueness is guaranteed
by  Propositions \ref{prop:sVwell} and \ref{prop:sWwell}. Then
\begin{align*}
(\Box + q_1) ( v_1 - v_2) = -(q_1-q_2) v_2, & \qquad t \geq |x|, ~ (x,t) \in \R^3 \times \R,
\\
(v_1-v_2)(x,|x|) = - \frac{1}{8 \pi} \int_0^1 (q_1-q_2)(sx) \, ds, & \qquad x \in \R^3,
\end{align*}
and
\begin{align*}
(\Box + q_1) ( w_1 - w_2) = -(q_1-q_2) w_2, & \qquad -|x| \leq t < |x|, ~ (x,t) \in \R^3 \times \R,
\\
(w_1-w_2)(x,-|x|) = - \frac{1}{8 \pi} \int_1^\infty (q_1 -q_2)(sx) \, ds, & \qquad x \in \R^3, ~ x \neq 0,
\\
(w_1 - w_2)(x,t)=0, & \qquad -|x| \leq t \leq -1.
\end{align*}
For $(x,t) \in K$, define 
\[
\vt_2(x,t) := \begin{cases} v_2(x,t), & t > |x| >0, \\ - w_2(x,-t), & -|x|<t < |x|, \end{cases}
\]
and
\[
\alpha(x,t) := \begin{cases} (v_1-v_2)(x,t), & \qquad t \geq |x|>0, \\
 - (w_1-w_2)(x,-t), & \qquad -|x|<t \leq |x|. \end{cases}
\]
Note that $\alpha$ has smooth extensions to the regions $t \geq |x|>0$, $-|x|<t \leq |x|$ and
\beqn
\alpha(x,t) =0, \qquad \text{when} ~~ 1 \leq t \leq |x|.
\label{eq:arzero}
\eeqn
because of (\ref{eq:swic}).
Also for $x \neq 0$
\begin{align}
\alpha(x,|x|+)  - \alpha(x,|x|-) 
& = - \frac{1}{8 \pi} \int_0^1 (q_1-q_2)(sx) \, ds - \frac{1}{8 \pi} \int_1^\infty (q_1-q_2)(sx) \, ds
\nn \\
& = - \frac{1}{8 \pi} \int_0^\infty (q_1-q_2)(sx) \, ds 
= - \frac{1}{8 \pi r} \int_0^\infty (q_1-q_2)(s \theta).
\label{eq:sacc1}
\end{align}
Now
\[
\Box \left ( \frac{H(t-|x|)}{|x|} \right ) =0, \qquad \text{for} ~ (x,t) \in K
\]
so, for $(x,t) \in K$, we have
\begin{align*}
\Box ( H(t-|x|) ) & = \Box \left ( |x| \, \frac{H(t-|x|)}{|x|} \right )
\\
& = \Box (|x|) \, \frac{H(t-|x|)}{|x|} - 2 \nabla (|x|) \cdot \nabla \left ( \frac{H(t-|x|)}{|x|} \right )
\\
& = - \frac{2}{r} \, \frac{ H(t-r)}{r} - 2 \frac{x}{r} \cdot \left ( - \frac{x}{r^3} H(t-r) - \frac{x}{r^2} \delta(t-r) \right )
\\
& = \frac{2}{r} \delta(t-r)
\end{align*}
and, again on $K$,
\begin{align*}
\Box ( H(|x|-t) ) = \Box (1 - H(t-|x|)) = - \frac{2}{r} \delta(t-r).
\end{align*}
Taking smooth extensions of $(v_1-v_2)(x,t)$ and $(w_1-w_2)(x,-t)$ to $K$, for $(x,t) \in K$ we have 
\begin{align*}
(\Box + q_1) \alpha(x,t) 
& = (\Box + q_1) ( (v_1-v_2)(x,t) H(t-|x|) - (w_1-w_2)(x,-t) H(|x|-t))
\\
& = -(q_1 - q_2)\, \vt_2 + (v_1-v_2)(x,t) \Box ( H(t-|x|)) - (w_1-w_2)(x,-t) \Box ( H(|x|-t) )
\\
& \qquad 
+ 2\delta(t-|x|) (\pa_t + \theta \cdot \nabla) (v_1-v_2)(x,t)
-2 \delta(|x|-t) ( ( \pa_t - \theta \cdot \nabla) (w_1-w_2) )(x,-t)
\\
& =  -(q_1 - q_2) \, \vt_2 + \frac{2}{r} \delta(t-r) ( (v_1-v_2)(r \theta,t) + (w_1-w_2)( r \theta,-t) )
\\
& \qquad + 2 \delta(t-r) ( \pa_t +  \theta \cdot \nabla)(v_1-v_2)(r \theta,r)
- 2 \delta(t-r) ( (  \pa_t - \theta \cdot \nabla) (w_1-w_2) ) (r \theta,-r)
\\
& =  -(q_1 - q_2) \, \vt_2 + \frac{2}{r} \delta(t-r) ( (v_1-v_2)(r \theta,r) + (w_1-w_2)(r \theta,-r) )
\\
& \qquad + 2 \delta(t-r)  \frac{d}{dr} ( (v_1-v_2)(r \theta, r) ) 
+ 2 \delta(t-r) \frac{d}{dr} ( (w_1-w_2)(r \theta, -r) )
\\
& =  -(q_1 - q_2) \, \vt_2 + \frac{2}{r} \delta(t-r) \frac{d}{dr} ( r \alpha(r \theta, r+) - r \alpha(r \theta, r-) )
\\
& =  -(q_1 - q_2) \, \vt_2,
\end{align*}
because of (\ref{eq:sacc1}).
Summarizing, $\alpha$ is smooth on the regions $t \geq |x|>0$ and $-|x|<t \leq |x|$ with
\begin{align}
(\Box+q_1) \alpha = - (q_1-q_2) \vt_2, & \qquad (x,t) \in K,
\label{eq:salphadede}
\\
\alpha(x,|x|+) = - \frac{1}{8 \pi} \int_0^1 (q_1 - q_2)(sx) \, ds, & \qquad x \neq 0,
\label{eq:salphacccc}
\\
\alpha(x,|x|+) - \alpha(x,|x|-) = -\frac{1}{8 \pi} \int_0^\infty (q_1 - q_2)(sx) \, ds,
& \qquad x \neq 0,
\label{eq:salphajj}
\\
\alpha(x,t)=0, & \qquad 1 \leq t \leq |x|.
\label{eq:safc}
\end{align}

If the $q_1, q_2$ are such that
\beqn
v_1|_{S \times [1,3]} = v_2|_{S \times [1,3]},
\qquad 
w_1|_{S \times [-1,1)} = w_2|_{S \times [-1,1)};
\label{eq:svwequal}
\eeqn
then $\alpha|_{S \times (-1,3]} =0$ and we show that $q_1=q_2$. 

Firstly, we claim $\alpha$ is continuous across $t=|x|$ on $K$. 
Since $v_1(x,t) = v_2(x,t)$ on $S \times [1,3]$, we have $v_1(x,|x|) = v_2(x,|x|)$ for all $x \in S$. Hence
\[
\int_0^1 q_1(sx) \, ds = \int_0^1 q_2(sx) \, ds, \qquad \forall x \in S,
\]
and since the $q_i$ are supported in $|x| \leq 1$, we have 
\[
\int_0^\infty (q_1-q_2)(sx) \, dx = 0, \qquad \forall x \in \R^3, ~ |x| = 1,
\]
which implies
\[
\int_0^\infty (q_1-q_2)(sx) \, dx = 0, \qquad \forall x \in \R^3
\]
and hence, from (\ref{eq:salphajj}), the jump in $\alpha$ across $t=|x|$ is 0.

Summarizing, $\alpha$ is smooth on the regions $t \geq |x|$ and $-|x|<t \leq |x|$, continuous across $t=|x|$ and 
satisfies
(\ref{eq:salphadede}), (\ref{eq:salphacccc}), with $\alpha|_{S \times (-1,3]} =0$. So, from Lemma 
\ref{lemma:sexterior},
we have $\pa_\nu \alpha|_{S \times (-1,3]}=0$, hence $q_1=q_2$ from Proposition \ref{prop:sprop}. Note the hypothesis of 
Lemma \ref{lemma:sexterior} holds because of (\ref{eq:salphadede}) and (\ref{eq:safc}).

%
%

\subsection{Proof of Proposition \ref{prop:sprop}}\label{subsec:sprop}

We define $K := \{ (x,t) \in \R^3 \times \R : -|x|<t \}$,
\[
Q := \{ (x,t) \in \R^3 \times \R : \ep \leq |x| \leq 1,\,  -|x| < t \leq 3. \}, \qquad 
Q_{+} := Q \cap \{ t \geq |x| \}.
\]
Suppose $p,q$ are supported in $\Bb$ and zero on $|x| \leq 3 \ep$, $\ep$ small.
Choose an $a$ between $4$ and $4(1- 9\ep^2)^{-1}$ and define
\[
\phi(x,t) = a|x|^2 - (t-|x|)^2, \qquad (x,t) \in \R^3 \times \R;
\]
from Lemma \ref{lemma:sweight} we know that 
\[
\psi = e^{\lambda \phi}
\]
is strongly pseudoconvex w.r.t $\Box$ in a neighborhood of $Q$ for a large enough $\lambda$.
For convenience, at times, we use the expressions $\phi(r,t)$ and $\psi(r,t)$ instead of 
$\phi(x,t)$ and $\psi(x,t)$; here $r=|x|$.


We claim (see Figure \ref{fig:chi}), that the smallest value of $\phi$ on the set $\{ (x,|x|) : 3 \ep \leq |x| \leq 1 \}$ is larger than the largest value 
of $\phi$ on 
\[
 \{(x,t) : |x| \leq 1, ~ t=3 \} \cup \{ (x,-|x|) : |x| \leq 1 \} \cup \{(x,t) : |x| \leq 2 \ep, ~ t \in \R \}.
\]
On $t=r$, $3\ep \leq r \leq 1$, the smallest value of $\phi$ is $9 a \ep^2$. 
On $t=3$, $0 \leq r \leq 1$, the largest value of $\phi$ is $a - 4$. 
On $t=-r$, the largest value of $\phi$ is $a-4$.
On $0 \leq r \leq 2 \ep$, $t \in \R$, the largest value of $\phi$ is bounded above by $4a \ep^2$. Hence our claim is proved because we chose $a$ between $4$ and $4/(1-9 \ep^2)$.

So we can find $\delta>0$, $c \in \R$ and a small $l$ in $(0, \ep)$ such that (see Figure \ref{fig:chi})
\vspace{-0.2in}
\begin{itemize}
\item $\psi \leq c - \delta$ on the set 
\[
\Bb \times [3-l,3] \cup \{ (x,t) : 0 \leq |x| \leq 1,  0 \leq t+r \leq l \} \cup 
\{ (x,t) : |x| \leq 2 \ep, -1 \leq t \leq 3 \};
\]
\item $\psi \geq c+\delta$ on $t=r$, $3\ep \leq r \leq 1$.
\end{itemize}
\vspace{-0.2in}
Choose $\chi(x,t)$, a compactly supported smooth function on $\R^3 \times \R$ such that 
$\chi$ is $1$ near $t=|x|$ and $0$ on the parts of $\pa Q$ where we do not have information. More specifically,
we construct a compactly supported smooth function $\chi$ such that (see Figure \ref{fig:chi})
\vspace{-0.2in}
\begin{itemize}
\item $\chi(x,t) =1$ on a neighborhood of $\{ (x,t) : t=r,  ~3 \ep \leq r \leq 1 \}$;
\item $\chi(x,t) =0$ when $|x| \leq \ep$ or when $t \geq 3 -l/2$;
\item $\chi(x,t) =0$ on $\{ (x,t): 0 \leq t+r \leq l/2, 0 \leq r \leq 1 \}$;
\item $\nabla_{x,t} \chi$ is non-zero only when $ \ep \leq |x| \leq 2 \ep$ or $ 3 -l \leq t \leq 3 - l/2$ or when
$l/2 \leq r+t \leq l$.
\end{itemize}
\begin{figure}
\begin{center}
\epsfig{file=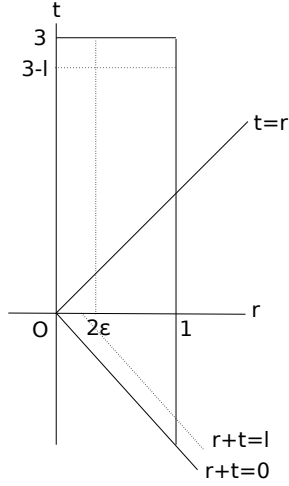, height=2.5in}
\end{center}
\caption{Construction of $\chi$}
\label{fig:chi}
\end{figure}
Define
\[
\beta(x,t) := \chi(x,t) \alpha(x,t), \qquad (x,t) \in K;
\]
$\beta$ has the same regularity properties as $\alpha$, $\beta$ and $\beta_r$ are zero on $|x|=\ep$ by construction
and $\beta$ and $\beta_r$ are zero on $S \times (-1,3]$ because of the hypothesis on $\alpha$. Further, 
$\beta, \beta_t$ are zero on $t=3$ and $\beta$ is zero in a neighborhood of $t=-r$ because of $\chi$.
Since $\beta \in H^1(K)$ is compactly supported and $\psi$ is strongly pseudoconvex w.r.t $\Box$ in a 
neighborhood of $Q$, from 
Theorem 1 in \cite{Ta96}, we have
\beqn
\sigma \int_Q e^{2 \sigma \psi} ( |\nabla_{x,t} \beta |^2 + \sigma^2 \beta^2 )
\cleq
\int_Q e^{2 \sigma \psi} |\L \beta|^2,
\label{eq:sLbeta}
\eeqn
for large $\sigma$. Here and below $\mathcal{L} = \Box + q$.

For $r \in [3 \ep, 1]$, using (\ref{eq:salphacc}), and that $\chi$ is $1$ in a neighborhood of $t=r$ with 
$3 \ep \leq r \leq 1$, we have
\begin{align*}
p(r \theta) & = \frac{d}{dr} \int_0^r p(s \theta) \, ds = \frac{d}{dr} \left ( r \int_0^1 p(rs \theta) \, ds \right ) 
\\
& = \frac{d}{dr} ( r \alpha( r \theta, r) ) = \frac{d}{dr} ( r \beta( r \theta, r) ) 
\\
& = (r \theta \cdot \nabla \beta + r\beta_t + \beta)(r \theta,r)
\\
& = (r (\beta_r + \beta_t) + \beta )(r \theta, r).
\end{align*}
Hence, using the support of $p$ and Lemma \ref{lemma:neartex} we have
\begin{align}
\int_B e^{2 \sigma \psi(x,|x|) } |p(x)|^2 \, dx
& =  \int_{B \setminus B_\ep}  e^{2 \sigma \psi(x,|x|) } |p(x)|^2 \, dx
\nn
\\
& \cleq  \int_{B\setminus B_\ep} e^{2 \sigma \psi(x,|x|)} ( |\beta_r + \beta_t)|^2 + \beta^2)(x,|x|) \, dx
\nn
\\
& \cleq \sigma \|\beta\|_{1,\sigma, Q_+}^2 + \|\L \beta \|_{0,\sigma, Q_+}^2
\\
& \cleq \int_Q e^{2 \sigma \psi} |\L \beta|^2,
\label{eq:stemp25}
\end{align}
with the last inequality a consequence of (\ref{eq:sLbeta}).
Noting that $\alpha$ is continuous across $t=|x|$ 
and smooth on each side of $t=|x|$ one may verify using a test function that
\begin{align*}
\L \beta &= \L ( \chi \alpha )  = \chi \L \alpha + \alpha \Box \chi + 2 ( \chi_t \alpha_t - \nabla \chi \cdot \nabla \alpha).
\end{align*}
So using the hypothesis of the proposition, we have
\[
|\L \beta| \cleq \chi \lvert p f \rvert + h \qquad \text{on} ~ Q
\]
for a bounded function $h$ on $Q$ with support on the region where $\nabla_{x,t} \chi$ is non-zero;
hence $\psi \leq c-\delta$ on the support of $h$. So, in (\ref{eq:stemp25}), using the function 
$g(\sigma)$ from Lemma \ref{lemma:sintegral}, we obtain
\begin{align*}
\int_B e^{2 \sigma \psi(x,|x|) } |p(x)|^2 \, dx
& \cleq \int_Q e^{2 \sigma \psi(x,t)} |p(x)|^2 \, dx \, dt + e^{2 \sigma (c-\delta)}
\\
& \cleq \int_{\ep \leq |x| \leq 1} e^{2 \sigma \psi(x,|x|)} |p(x)|^2 \int_{-1}^3 e^{ \psi(x,t) - \psi(x,|x|)} \, dt \, dx
+ e^{2 \sigma (c-\delta)}
\\
& \cleq g(\sigma) \int_B e^{2 \sigma \psi(x,|x|)} |p(x)|^2 + e^{2 \sigma (c-\delta)}.
\end{align*}
Hence, from Lemma \ref{lemma:sintegral}, for large enough $\sigma$ we have
\[
\int_B e^{2 \sigma \psi(x,|x|) } |p(x)|^2 \, dx
\cleq e^{2 \sigma (c-\delta)}.
\]
Since $\psi(x,|x|) \geq c+\delta$ on $3 \ep \leq |x| \leq 1$ and $p(x)$ is supported in $3 \ep \leq |x| \leq 1$, we obtain
\[
e^{2\sigma(c+\delta)} \int_B |p(x)|^2 \, dx 
\cleq e^{2 \sigma (c-\delta)}
\]
which is equivalent to 
\[
\int_B |p(x)|^2 \, dx \cleq e^{-4 \sigma \delta},
\]
for large enough $\sigma$. So letting $\sigma \to \infty$ we conclude that
\[
\int_B |p(x)|^2 \, dx =0
\]
and hence $p=0$ on $B$.

%
%

\subsection{Proof of Proposition \ref{prop:sWwell}}

Choose a $\chi(x) \in C^\infty(\R^3)$ which is supported in $|x| \geq 1/4$ and $1$ on $|x| \geq 1/2$. 
We seek $W(x,t)$ in the form
\[
W(x,t) = \chi(x) \, \frac{ \delta(t+|x|)}{4 \pi |x|} + \Wbv(x,t)
\]
so we need to prove the well-posedness of the inhomogeneous IVP
\begin{align}
(\Box + q) \Wbv = F(x,t), & \qquad (x,t) \in \R^3 \times \R,
\label{eq:Wbvde}
\\
\Wbv(x,t)=0, & \qquad  t< -1,
\label{eq:Wbvic}
\end{align}
where
\beqn
F(x,t) = - (\Box +q) \left ( \chi(x) \, \frac{ \delta(t+|x|)}{4 \pi |x|} \right ).
\label{eq:sFdef}
\eeqn
Later we show that $F(x,t)=0$ for $t<-1$ and $F(x,t) \in L^2 ( \R, H^s(\R^3) )$ for all $s<-7/2$.
So, from Theorems 9.3.1, 9.3.2 and the remark after that in \cite{Ho76}, we conclude
that (\ref{eq:Wbvde}), (\ref{eq:Wbvic}) has a unique solution in the class of functions which are locally 
in $H^1( \R, H^s(\R^3))$, $s<-7/2$.

Next we address the regularity of $W(x,t)$. We have
\[
W(x,t) = \frac{ \delta(t+|x|)}{4 \pi |x|}, \qquad t<-1
\]
and, for $t<-1$, the wave front set of $W$ is
\begin{align*}
\text{WF} \left ( \frac{ \delta(t+|x|)}{4 \pi |x|} \right )
&  \subset \Gamma =  \{ (x, -|x|:  \sigma x, \sigma |x| ) \; :  ~ x \in \R^3, ~ |x| \geq 1, ~ \sigma \in \R \}.
\end{align*}
Since $q$ is smooth and $(\Box +q )W=0$ on $\R^3 \times \R$, from H\"ormander's propagation of singularities theorem
(Theorem 2.1 in Chapter 6 of \cite{Ta81}), the wave front set of $W$ is invariant w.r.t the bicharacteristic flow 
associated with $\Box$, hence singularities of $W$ are preserved along rays of $\Box$. 
Further, for $t<-1$, the singularities of $W$ must lie on $t=-|x|$. Since the $x,t$ rays are lines which make a 45 degree
angle with lines parallel to the $t$ axis, the only $x,t$ rays which lie on 
$t=-|x|$ for $t<-1$ are those which lie on the cone $t^2= |x|^2$. Hence the singularities of $W$ lie on $t^2=|x|^2$ only.


Since $W(x,t)$ is supported on $t = -|x|$ when $t<-1$, from a domain of dependence argument we see that $W$ is 
supported in $ t \geq -|x|$. On the region $t<|x|$, we seek $W(x,t)$ in the form
\[
W(x,t) = \frac{1}{4 \pi} \frac{ \delta(t+|x|)}{4 \pi |x|} + w(x,t) H(t+|x|)
\]
where $w(x,t)$ will be a smooth function on $-|x| \leq t < |x|$.

First (\ref{eq:sWic}) forces $w(x,t)=0$ for $-|x| \leq t<-1$. Next, since 
\[
\Box \, \left (  \frac{\delta(t+|x|)}{4 \pi |x|}  \right ) 
 = 0 \qquad \text{when} ~ |x|<t,
\]
(\ref{eq:sWde}) forces
\beqn
(\Box + q) (w(x,t) H(t+|x|) ) = - q(x) \, \frac{\delta(t+|x|)}{4 \pi |x|} \qquad \text{when} ~ t<|x|.
\label{eq:swwant}
\eeqn
When $x \neq 0$ we have (below $n=3$) 
\begin{align*}
\nabla( w(x,t) H(t+|x|))  & = \frac{x}{|x|} w(x,t) \delta(t+|x|) + (\nabla w)(x,t) H(t+|x|),
\\
\Delta ( w(x,t) H(t+|x|)) & = w(x,t) \delta'(t+|x|) + 2 \frac{ x}{|x|}  \cdot \nabla w(x,t) \, \delta(t+|x|) 
\\
& \qquad + \frac{n-1}{|x|} w(x,t) \delta(t+|x|) + (\Delta w)(x,t) H(t+|x|),
\\
 \pa_t ( w(x,t) H(t+|x|) ) & = w(x,t) \delta(t+|x|) + w_t (x,t) H(t+|x|),
 \\
 \pa_t^2 ( w(x,t) H(t+|x|) ) & = w(x,t) \delta'(t+|x|) + 2w_t (x,t) \delta(t+|x|) + w_{tt}(x,t) H(t+|x|).
\end{align*}
Hence, for $x \neq 0$
\[
(\Box + q) (w(x,t) H(t+|x|) ) = 2 \left ( w_t - \frac{x}{|x|} \cdot \nabla w - \frac{w}{|x|} \right  ) \delta(t+|x|)
+ (\Box + q)(w(x,t) ) H(t+|x|).
\]
Using this in (\ref{eq:swwant}) we see that we need to find a smooth $w(x,t)$ on $-|x| \leq t < |x|$ such that
\beqn
(\Box + q) w =0 \qquad \text{on} ~ -|x| \leq t < |x|
\label{eq:wwwde}
\eeqn
and
\beqn
(-|x| w_t + x \cdot \nabla w + w)(x,-|x|) =  \frac{q(x)}{8 \pi}, \qquad x \neq 0.
\label{eq:swtemp1}
\eeqn
Now (\ref{eq:swtemp1}) may be written as 
\[
\frac{d}{dr}( r w(r \theta, -r)) =  \frac{ q(r \theta)}{8 \pi}, \qquad r>0, ~ |\theta|=1,
\]
which combined with $w(r \theta,t)=0$ for $t<-1$ implies
\[
r w(r \theta, -r) = - \frac{1}{8 \pi} \int_r^\infty q(s \theta) \, ds, \qquad r>0, ~ |\theta|=1
\]
and hence
\beqn
w(x,-|x|) = - \frac{1}{8 \pi} \int_1^\infty q(sx) \, ds, \qquad x \neq 0.
\label{eq:wwwcc}
\eeqn
The existence and uniqueness of a smooth $w(x,t)$ on $-|x| \leq t < |x|$ satisfying (\ref{eq:wwwde}), (\ref{eq:wwwcc})
and $w(x,t)=0$ on $-|x| \leq t < -1$ is proved in \cite{Ba18}.


We now prove the earlier claim that the $F(x,t)$ defined by (\ref{eq:sFdef}) is in $ L^2(\R, H^s(\R^3))$ for all 
$s<-7/2$. 
From Theorem 7.3.1 in \cite{FJ98} we have
\[
\Box \left ( \frac{ \delta(t+|x|)}{4 \pi |x|}  \right ) =\delta(x,t),
\]
and since $\chi(x)$ supported in $|x| \geq 1/4$, $\nabla \chi(x)$ is supported in $ 1/4 \leq |x| \leq 1/2$, 
$q(x)$ is supported in $|x| \leq 1$, we conclude that
\begin{align*}
F(x,t)  & =  \frac{\delta(t+|x|)}{4 \pi |x|} \, (\Delta + q) (\chi(x) ) - 
2 \nabla \chi(x) \cdot \nabla \left (  \frac{ \delta(t+|x|)}{4 \pi |x|}  \right )
\\
& = a(x) \delta(t+|x|) + b(x) \, \delta'(t+|x|)
\end{align*}
where $a(x), b(x)$ are smooth functions supported in $1/4 \leq |x| \leq 1$.
Now, for $(\xi,\tau) \in \R^3 \times \R$, the Fourier transform of $b(x) \delta'(t+|x|)$ is
\begin{align*}
\F ( b(x) \, \delta'(t+|x|) )(\xi, \tau)
& =  \int_{\R^3 \times \R} b(x) \delta'(t+|x|) \, e^{-i x \cdot \xi - i  t \tau} \, dx \, dt
\\
& =  i \tau \int_{\R^3} b(x) \, e^{-i x \cdot \xi} \,  e^{i |x| \tau} \, dx,
\end{align*}
implying
\[
| \F (  b(x) \, \delta'(t+|x|) ) (\xi, \tau) | \leq C |\tau|. \qquad \text{(used when $\tau$ small)}
\]
Further, for $\tau \neq 0$, we have
\begin{align*}
\F ( b(x) \, \delta'(t+|x|) )(\xi, \tau) 
& = - \frac{\tau}{\tau^2} \int_{\R^3} b(x)  e^{-i x \cdot \xi} \,  \pa_r^2 ( e^{i |x| \tau}  ) \, dx
\\
& = - \frac{1}{\tau} \int_{\R^3} r^{-2}  \pa_r^2 ( r^2 b(x)  e^{-i x \cdot \xi} )  \; e^{i |x| \tau}   \, dx
\end{align*}
implying
\[
| \F (  b(x) \, \delta'(t+|x|) ) (\xi, \tau) | \leq C \,  \frac{ (1+|\xi|)^2}{|\tau|}. \qquad \text{(used when $\tau$ not small)}
\]
Hence, for $s<0$, using the above upper bounds we have
\begin{align*}
\| b(x) \, \delta'(t+|x|)   \|_{L^2(\R, H^s(\R^3))}^2
& = \int_\R \int_{\R^3} (1+|\xi|^2)^{s} |\F ( b(x) \, \delta'(t+|x|) )(\xi, \tau)|^2 \, d \xi \, d \tau
\\
&= \left ( \int_{-1}^1 \int_{\R^3} + \int_{|\tau| \geq 1} \int_{\R^3} \right )
(1+|\xi|^2)^{s} |\F ( b(x) \, \delta'(t+|x|) )(\xi, \tau)|^2 \, d \xi \, d \tau 
\\
& 
\cleq \int_{-1}^1 \tau^2 \int_{\R^3} (1+|\xi|^2)^{s} \, d \xi \, d \tau 
+ \int_{|\tau| \geq 1} \int_{\R^3} (1+|\xi|^2)^{s}  \frac{ (1+|\xi|^2)^2}{\tau^2} \, d \xi \, d \tau 
\\
& \cleq \int_{\R^3} (1+|\xi|^2)^{s} \, d \xi  + \int_{\R^3} (1+|\xi|^2)^{s}  (1+|\xi|^2)^2 \, d \xi
\\
& \cleq \int_0^\infty r^2 (1+r^2)^s \, dr + \int_0^\infty r^2 (1+r^2)^{s+2} \, dr
\end{align*}
is finite if $ 2s+6<-1$, that is if $s<-7/2$. Hence $a(x) \delta'(t+|x|) \in L^2(\R,H^s(\R^3))$ if $s<-7/2$. Using
a similar argument one may show that $b(x) \delta(t+|x|) \in L^2(\R,H^s(\R^3))$ for at least $s<-7/2$.




\section{Lemmas for the spherical and point source problem}\label{sec:slemma}


For this section, define $K := \{(x,t) \in \R^3 \times \R : -|x| < t, ~ x \neq 0 \}$, for $\ep>0$ we define 
\begin{gather*}
B_\ep := \{ x \in \R^3 : |x| < \ep \}, ~~
Q := \{ (x,t) \in \R^3 \times \R : \ep \leq |x| \leq 1,\,  -|x| < t \leq 3. \},
~~
\Sigma := S \times (-1,3],
\end{gather*}
and, for $\tau>0$, we define
\begin{align*}
Q_{+} := Q \cap \{ t \geq |x| \}, &   \qquad Q_-:= Q \cap \{ t \leq |x| \}, & 
Q_{+,\tau}= Q_+ \cap \{  t \leq \tau \},
\\
\Sigma_+ := \Sigma \cap \{ t \geq |x| \}, &  \qquad \Sigma_- := \Sigma \cap \{ t \leq |x| \}.
& 
\end{align*}

%
%

\subsection{Energy estimates for the spherical and point source problem}\label{subsec:senergy}

We derive a weighted energy estimate on $t=r$ and an energy estimate for the exterior problem, the first needed in
the proof of Proposition \ref{prop:sprop} and the second needed in the proof of Theorem \ref{thm:pointspherical}.

\def\vt{{\tilde{v}}}

\begin{lemma}[Energy estimate on $t=|x|$]\label{lemma:neartex}
If $q$ is a smooth function on $\R^3$ with support in $\Bb$ 
and $\beta, \psi$ are smooth functions on $t \geq |x|>0$ such that
$\beta$, $ \pa_r \beta$ are zero when $|x|=\ep$ ($\ep>0$ is small) or when $|x|=1$,  then 
\[
\int_{B \setminus B_\ep} e^{\sigma \psi(x,|x|)} 
\left ( ( \beta_t + \beta_r)^2 
+ \frac{1}{2 r^2} \smash{\sum_{ij} }(\Omega_{ij} \beta)^2 + \sigma^2 \beta^2 
\right  )(x,|x|) \, dx
\cleq
\sigma  \|\beta\|_{1,\sigma, Q_+}^2
+ \| (\Box+q) \beta \|_{0, \sigma, Q_+}^2
\]
for all $\sigma>0$ large enough, with the constant dependent only on $\ep$, $\|q\|_\infty$ and $\|\psi\|_{C^2(\ol{Q}_+)}$.
\end{lemma}


\begin{proof}

For $\sigma>1$, let $\mu = e^{\sigma \psi} \beta$; then on $Q_+$
\begin{align}
|\mu|  \cleq e^{ \sigma \psi} |\beta|, \qquad
|\nabla_{x,t} \mu|  \cleq e^{\sigma \psi} ( |\nabla_{x,t} \beta| + \sigma |\beta|).
\label{eq:smu2}
\end{align}
Further, since $\beta$ is smooth on the region $t \geq |x|$, we have
\begin{align*}
\L \mu & :=(\Box +q) \mu  = e^{\sigma \psi} \left ( \L \beta
+ 2 \sigma (\psi_t \beta_{t} -  \nabla \psi \cdot \nabla \beta)
+ \sigma \beta \Box \psi + \sigma^2 (\psi_t^2 - |\nabla \psi|^2) \beta \right )
\end{align*}
which implies
\beqn
|\L \mu| \cleq e^{\sigma \psi} ( |\L \beta| + \sigma |\nabla_{x,t} \beta| + \sigma^2 |\beta| ).
\label{eq:smu3}
\eeqn

Define
\[
J := \int_{B \setminus B_\ep} 
\left ( ( \mu_t + \mu_r)^2 
+ \frac{1}{2 r^2} \smash{\sum_{i \neq j} }(\Omega_{ij} \mu)^2 + \sigma^2 \mu^2 
\right  )(x,|x|) \, dx
\]
and, for any $\tau \in [1,3]$, define 
\[
E(\tau) := \int_{B \setminus B_\ep}  ( |\nabla_{x,t} \mu|^2 + \sigma^2 \mu^2)(x,\tau) \, dx.
\]
Note that from (\ref{eq:fangular}) we have
\[
|\nabla_{x,t} \mu|^2 +  2 \mu_r \mu_t = (\mu_t + \mu_r)^2 + \frac{1}{2 r^2}  \sum_{i \neq j}  (\Omega_{ij} \mu)^2.
\]

For any $\tau \in [1,3]$, integrating the relation
\[
2\mu_t (\Box \mu + \sigma^2 \mu) = ( \mu_t^2 + |\nabla \mu|^2 + \sigma^2 \mu^2)_t - 2 \nabla \cdot (\mu_t \nabla \mu)
\]
over the region $ Q_{+,\tau}$, and noting that $\mu$  and $\pa_r \mu$ are zero when $|x|=\ep$ or $|x|=1$,
we have
\begin{align*}
J & = E(\tau) - 2\int_{ Q_{+,\tau} } \mu_t (\Box \mu + \sigma^2 \mu )
 \\
 & = E(\tau) - 2\int_{ Q_{+,\tau} }  \mu_t ( \L \mu  + (\sigma^2+ q) \mu )
 \\
 &
 \cleq E(\tau)  +  \int_{Q_+} (| \L  \mu| + \sigma^2 |\mu|) |\mu_t|.
\end{align*}
Integrating this over $\tau \in [1,3]$, using (\ref{eq:smu2}) - (\ref{eq:smu3}) and that $\sigma^2 |\mu_t| \, |\mu|
\cleq \sigma |\mu_t|^2 + \sigma^3 |\mu|^2$ we obtain
\begin{align}
J & \cleq \sigma \int_{Q_+}  |\nabla_{x,t} \mu|^2 + \sigma^2 \mu^2 + 
  \int_{Q_+} | \L \mu| \, |\mu_t|
  \nonumber
  \\
  & \cleq  \sigma \int_{Q_+} e^{2 \sigma \psi} ( |\nabla_{x,t} \beta|^2 + \sigma^2 |\beta|^2)
+ \int_{Q_+} e^{2 \sigma \psi} |\L \beta|^2 .
 \label{eq:sE0}
\end{align}
Now
$\mu = e^{\sigma \psi} \beta$ so $\beta = e^{-\sigma \psi} \mu$ and hence, on $Q_+$, we have
\[
e^{2 \sigma \psi} |\beta_r + \beta_t|^2  \cleq |\mu_r + \mu_t|^2 + \sigma^2 |\mu|^2,
\qquad
e^{2 \sigma \psi} |\Omega_{ij} \beta|^2 \cleq |\Omega_{ij} \mu|^2 + \sigma^2 |\mu|^2,
\qquad
e^{2 \sigma \psi} |\beta|^2 = |\mu|^2.
\]
So this combined with (\ref{eq:sE0}) and the definition of $J$ proves the lemma.
\end{proof}


Next we obtain a uniqueness result for an exterior problem.

\begin{lemma}[Uniqueness for exterior problem]\label{lemma:sexterior}
Suppose $\alpha(x,t)$ is a continuous function on $K$ which is smooth on the subregions $t \geq |x|$ and $t \leq |x|$ and 
satisfies
\begin{align*}
\Box \alpha =0, & \qquad \text{on} ~ (x,t) \in K, ~ |x| \geq 1,
\\
\alpha(x,t) = 0, & \qquad \text{on} ~1 \leq t \leq |x|,
\\
 \alpha|_\Sigma=0;
\end{align*}
then $\alpha=0$ on the region $\{(x,t) : |x| \geq 1, ~ -1<t \leq 3 \}.$
\end{lemma}


\begin{proof}


Firstly $\alpha=0$ on the conical type region $ C: = \{ (x,t) : |x| \geq 1, ~ |x| \leq t \leq 3 \}$. This follows from 
integrating the identity
\[
2 u_t \Box u = (u_t^2 + |\nabla u|^2)_t - 2 \nabla \cdot (u_t \nabla u)
\]
over the region $ C \cap \{ t \leq \tau \}$, for all $\tau \in [1,3]$, and observing that $\alpha$ is smooth $C$ and
$\alpha=0$ on $\Sigma_+$ (the lateral boundary of $C$) and the conical boundary of $C$. 
Note the conical boundary of $C$ is a characteristic surface so $\alpha=0$ on this surface is adequate for our purpose.
Hence $\alpha=0$ on the region $\{(x,t): |x| \geq 1, ~ 1 \leq t \leq 3 \}$. 

Next, on the region
$\{ (x,t) : |x| \geq 1, ~ -1<t \leq 1 \}$, $\alpha$ is smooth and solves the backward IBVP $\Box \alpha=0$
with $\alpha=0$ on $\Sigma_-$ (the lateral boundary) and $\alpha=0, \alpha_t=0$ on $t=1$. So again by standard estimates
$\alpha=0$ on this region.

\end{proof}

%
%

\subsection{Carleman weight for the spherical wave problem}\label{subsec:scarl}

Please refer to beginning of subsection \ref{subsec:pcarl} for the definition of pseudoconvexity and strong pseudoconvexity 
for differential operators and associated results that we use here.

Our goal is to construct a function, dependent only on $r,t$, which is decreasing in 
$|t-r|$ for a fixed $r$ and strongly pseudoconvex w.r.t $\Box$. As discussed at the beginning
of subsection \ref{subsec:pcarl}, one starts by constructing a function whose level surfaces are pseudoconvex w.r.t
$\Box$. We start by
characterizing all functions, dependent only on $r,t$, whose level surfaces are pseudoconvex w.r.t $\Box$;
this may be useful elsewhere.

\begin{lemma}\label{lemma:scarl}
If $\phi(r,t)$ is a smooth function on $(0,\infty) \times \R$ such that $(\phi_r, \phi_t) \neq (0,0)$
at every point on this region, then the level curves of $\phibv(x,t) = \phi(|x|,t)$ are strongly pseudoconvex w.r.t 
$\Box$ on the region $ (\R^n \setminus \{0\}) \times \R$, $n>1$, iff the following holds on  $(0,\infty) \times \R$:
\beqn
r (\phi_{tt} \phi_r^2 + \phi_{rr} \phi_t^2 - 2 \phi_{rt} \phi_r \phi_t) + \phi_r(\phi_r^2 - \phi_t^2)
 >0,
\qquad \text{whenever} ~ |\phi_t| \leq |\phi_r| 
\label{eq:rtconvex}
\eeqn
\end{lemma}

\begin{proof}

As discussed at the beginning of subsection \ref{subsec:pcarl}, since $\Box$ is a second order operator with real principal symbol,
the level surfaces of $\phibv$ will be strongly pseudoconvex w.r.t $\Box$ on a region iff the level surfaces of $\phibv$ are
pseudoconvex w.r.t $\Box$ on that region.

Define $\Omega := (\R^n \setminus \{0\}) \times \R$,  $ \phibv(x,t) := \phi(|x|,t)$ on $\Omega$ and note that 
$\nabla_{x,t} \phibv \neq 0$ at every point of $\Omega$. 
Below double indices imply summation. Temporarily we denote $t$ by $x_0$, $\tau$ by $\xi_0$
and take $x=(x_0,x_1, \cdots, x_n)$, $\xi=(\xi_0, \cdots, \xi_n)$. 
The condition (1.3) in \cite{Ta96}, in expanded form, is condition (8.4.5) in \cite{Ho76}, so
the level surfaces of $\phibv$ are pseudoconvex w.r.t $P(x,D)$ (with principal symbol $p(x,\xi)$) on $\Omega$ iff
\[
\frac{\pa^2 \phibv}{\pa x_j \pa x_k } (x) \, \frac{ \pa p}{\pa \xi_j}(x, \xi) \, \frac{ \pa p}{\pa \xi_k}(x, \xi)
+ \left ( \frac{ \pa^2 p}{\pa \xi_j \pa x_k}(x, \xi) \, \frac{ \pa p}{\pa \xi_k}(x, \xi) - \frac{ \pa p}{\pa x_k} (x, \xi)
\frac{ \pa^2 p}{\pa \xi_j \pa \xi_k}(x, \xi) \right ) \frac{\pa \phibv}{\pa x_j}(x) \, > \, 0
\]
whenever $(x,\xi) \in \Omega \times \R^{n+1}$, $\xi \neq 0$ and
\[
p(x,\xi)=0, ~~ (\nabla_\xi p \cdot \nabla_x \phibv)(x,\xi) =0.
\]
If we introduce $\ep_0=-1$ and $\ep_j=1$ for $j=1, \cdots, n$ then the principal symbol of $\Box$ is
\[
p(x,\xi) =  \ep_j \xi_j^2
\]
and the psuedo-convexity condition may be rewritten as
\[
 \ep_j \ep_k \phibv_{x_j x_k} \xi_j \xi_k >0
\]
whenever $x \in \Omega$, $ \xi \neq 0$ and 
\[
 \ep_j \xi_j^2=0, ~~ \ep_j \xi_j \phibv_{x_j} =0.
\]
Written in the original variables, the pseudoconvexity condition is
\[
\tau^2 \phibv_{tt} - 2\tau \xi_j \phibv_{x_j t} + \xi_j \xi_k \phibv_{x_j x_k} >0, \qquad \forall (x,t) \in \Omega, 
~ (\xi, \tau) \neq (0,0), ~
\tau^2 = |\xi|^2, ~~ \tau \phibv_t = \xi \cdot \nabla_x \phibv.
\]
Since the condition is homogeneous in $(\xi,\tau)$, we may take $\tau = \pm 1$, $|\xi|=1$
and the condition is equivalent to 
\[
\phibv_{tt} - 2 \xi_j \phibv_{x_j t} + \xi_j \xi_k \phibv_{x_j x_k} >0
\]
whenever $(x,t) \in \Omega$, $|\xi|=1$ and $\phibv_t = \xi \cdot \nabla_x \phibv$.

Now
\begin{gather*}
\phibv_t = \phi_t,  ~~ \phibv_{x_j} = \phi_r \frac{x_j}{r}
\\
\phibv_{tt} = \phi_{tt}, ~~ \phibv_{x_j t} = \phi_{rt} \frac{x_j}{r}, 
~~ \phibv_{x_j x_k} = \phi_{rr} \frac{x_j x_k}{r^2} + 
\frac{\phi_r}{r} \delta_{jk} - \phi_r \frac{x_j x_k}{r^3}.
\end{gather*}
With $(\phi_r, \phi_t) \neq (0,0)$ at every point, the pseudoconvexity 
condition is
\begin{align*}
 \phi_{tt} - 2 \phi_{rt} \frac{(x \cdot \xi)}{r} + \phi_{rr} \frac{ (x \cdot \xi)^2}{r^2} + \phi_r \frac{ |\xi|^2}{r} - 
 \phi_r \frac{ (x \cdot \xi)^2}{r^3}  >0
\end{align*}
at points where $|\xi|^2=1$ and $ \phi_t = \phi_r (x \cdot \xi)/r$. Since $(\phi_t, \phi_r) \neq (0,0)$ at every point, the 
condition $ \phi_t = \phi_r (x \cdot \xi)/r$ holds only at points where $\phi_r \neq 0$, so at such points we can write
\[
\frac{x \cdot \xi}{r} = \frac{ \phi_t}{\phi_r}
\]
and hence the pseudoconvexity inequality is
\begin{align*}
0&< \phi_{tt} - 2 \phi_{rt} \frac{\phi_t}{\phi_r} + \phi_{rr} \frac{ \phi_t^2}{\phi_r^2} + \frac{\phi_r}{r} - \frac{ \phi_r \phi_t^2}{r \phi_r^2}
\\
& = \frac{ r (\phi_{tt} \phi_r^2 + \phi_{rr} \phi_t^2 - 2 \phi_{rt} \phi_r \phi_t) + \phi_r(\phi_r^2 - \phi_t^2)}{r \phi_r^2}.
\end{align*}
Also, since $|\xi|=1$, we have 
\[
\frac{|\phi_t|}{|\phi_r|} = \frac{ |x \cdot \xi|}{r} \leq \frac{|x| \, |\xi|}{r} \leq 1.
\]
Further, if $|\phi_t| \leq |\phi_r|$ at some point then we can find an 
$\xi \in \R^n$ with $|\xi|=1$ 
so that $\phi_t = \phi_r (x \cdot \xi)/r$ at that point. 
Hence we have proved the lemma.
%
\end{proof}


Next we construct a function of $r,t$ which is strongly 
pseudoconvex w.r.t $\Box$ and such that the function is a decreasing function of $|t-r|$ for a fixed $r$.

\begin{lemma}\label{lemma:sweight}
For $a>4$, if
\[
\phi(x,t) := a |x|^2 - (t-|x|)^2
\]
then
\[
\psi(x,t) := e^{\lambda \phi(x,t)},
\]
for large enough $\lambda$, is strongly pseudoconvex w.r.t $\Box$ on the region $(\R^3 \setminus \{0\}) \times \R$. 
\end{lemma}


\begin{proof}

From the discussion at the beginning of subsection \ref{subsec:pcarl} it is enough to prove that the level surfaces of $\phi$
are strongly pseudoconvex w.r.t $\Box$ on the region $(\R^3 \setminus \{0\}) \times \R$, so we use
Lemma \ref{lemma:scarl} to prove this.

For convenience we take
\[
\phi(r,t) := \frac{1}{2} (a r^2 - (t-r)^2) = \frac{1}{2} ( (a-1)r^2 + 2rt - t^2)  = \frac{1}{2} (b r^2 + 2rt - t^2)
\]
where $b=a-1$. Hence
\begin{align*}
\phi_r = br + t, \qquad \phi_t = r-t
\\
\phi_{rr} = b, \qquad \phi_{rt} = 1, \qquad \phi_{tt}=-1.
\end{align*}
Now $\phi_r = 0$ and $\phi_t=0$ exactly at the points where $r=t$ and $br+t=0$, that is iff  $b=-1$, because
we are working in the region where $r \neq 0$. So we have to be sure that $b \neq -1$.

Next, the condition $| \phi_t| \leq |\phi_r|$ is equivalent to 
\[
|r-t|^2 \leq |br+t|^2
\]
which simplifies to
\[
(1-b^2)r \leq 2(b+1)t.
\]
If we choose to have $b+1>0$ (that is $a>0$) then the condition $|\phi_t| \leq |\phi_r|$ is equivalent to
\beqn
t \geq \frac{1-b}{2} r.
\label{eq:snec}
\eeqn
Next, from Lemma \ref{lemma:scarl}, for $\phi$ to have strongly pseuodoconvex level surfaces w.r.t $\Box$, we want
\begin{align*}
0  & <  (r \phi_{tt} + \phi_r) \phi_r^2 + (r \phi_{rr} - \phi_r)\phi_t^2 - 2 r \phi_{rt} \phi_r \phi_t
\\
& =( -r + br +t)( br+t)^2 + (br - br -t) (r-t)^2 - 2r (r-t)(br+t) 
\\
& = ( (b-1)r +t)( br +t)^2 -t (r-t)^2 - 2r (r-t)(br+t)
\\
& = ( (b-1)r +t)( br +t)^2 - (r-t) ( tr - t^2 + 2br^2 + 2rt)
\\
& = ( (b-1)r +t)( br +t)^2 - (r-t) ( 2br^2 + 3rt -t^2)
\\
& = 0 t^3 + rt^2 ( (b-1) + 2b + 1+3) + r^2t( 2b (b-1) + b^2 + 2b -3)  + r^3( b^2(b-1) -2b)
\\
& = 3rt^2(b+1) + 3tr^2(b^2-1) + b (b^2-b-2) r^3
\\
& = (b+1)r ( 3t^2 + 3(b-1)tr + b(b-2)r^2)
\end{align*}
whenever (\ref{eq:snec}) holds.
So assuming we choose $b+1>0$, we want the quadratic form
\[
f(r,t) := r^2 b(b-2) + 3(b-1)rt + 3t^2
\]
to be positive in the region $t > (1-b) r/2$. But $(r,t)$ and $(-r,-t)$ give the same value of the quadratic form but opposite
inequalities for $t > (1-b) r/2$, so we want this quadratic form to be positive definite. 
If we choose $b>2$ then this 
quadratic form will be positive definite if 
\[
0 <12b(b-2) - 9(b-1)^2 = 3b^2 -6b -9 = 3((b-1)^2 -4),
\]
that is if $b>3$. So we conclude that if $b>3$ then the level curves of $\phi(r,t)$ are strongly pseudoconvex in the
region where $r \neq 0$.
\end{proof}


Next, we compute the limit of an integral associated with the Carleman weight we use in the proof of
Proposition \ref{prop:sprop}.

\begin{lemma}\label{lemma:sintegral} 
Let $D = \{ x \in R^3 : \ep \leq |x| \leq 1 \}$ and $\psi(x,t)$ be the function in Lemma \ref{lemma:sweight}. Define
\[
g(\sigma) :=  \sup_{ x \in D} \int_{-1}^3 e^{2\sigma ( \psi (x,t) - \psi(x,|x|)) }\, dt ;
\]
then $\lim_{\sigma \to \infty} g(\sigma) =0.$
\end{lemma}


\begin{proof}
Below $x \in D$. Since $\lambda >0$ and $\phi(x,|x|) \geq 0$, we have
\begin{align*}
\psi(x,|x|) - \psi(x,t)   & = e^{\lambda \phi(x,|x|)} - e^{\lambda \phi(x,t)}
 =  e^{\lambda \phi(x,|x|)} (1 -  e^{ - \lambda (\phi(x,|x|) -  \phi(x,t) ) } )
\\
& = e^{\lambda \phi(x,|x|)} ( 1- e^{ - \lambda (t-|x|)^2})
 \geq 1- e^{ - \lambda (t-|x|)^2}.
\end{align*}
Now, for $s \geq 0$, 
\[
1 - e^{-s} \geq \min( 1/2, s/2),
\]
hence
\[
2(\psi(x,t) - \psi(x,|x|)) \leq - \min(1, \lambda(t-|x|)^2 )
\]
so
\begin{align*}
\int_{-1}^3 e^{2 \sigma (\psi(x,t)-\psi(x,|x|)} \, dt
& \leq \int_{-1}^3 e^{ - \sigma \min( 1, \lambda (t-|x|)^2) } \, dt
 = \int_{-1-|x|}^{3-|x|} e^{ - \sigma \min( 1, \lambda t^2) } \, dt
\\
& \leq \int_{-2}^{3} e^{ - \sigma \min( 1, \lambda t^2) } \, dt.
\end{align*}
Hence, by the dominated convergence theorem
\[
0 \leq  \lim_{\sigma \to \infty} g(\sigma) \leq 
\lim_{\sigma \to \infty}  \int_{-2}^{3} e^{ - \sigma \min( 1, \lambda t^2) } \, dt =0.
\]
\end{proof}



\section{Acknowledgment}

This work was done when Rakesh was on sabbatical from the University of Delaware, mainly at the University of Helsinki but
also at University of Jyv{\"a}skyl{\"a}. Rakesh would like to thank the University of Helsinki, particularly Matti Lassas, for its generous support and also University of Jyv{\"a}skyl{\"a} for its support. Rakesh was also supported by funds from an NSF grant DMS 1615616. 
M Salo's work was supported by the Academy of Finland (grants 284715 and 309963) and by the European Research Council under Horizon 2020 (ERC CoG 770924).

%

\end{document}